\def\mymode{x}  

\if\mymode i
\documentclass{amsart}
\else
\documentclass[11pt]{amsart}
\fi

\usepackage[breaklinks]{hyperref}
\usepackage{amsmath}
\usepackage{amssymb}
\usepackage{mathrsfs}
\usepackage{amsthm}
\usepackage[utf8]{inputenc}
\usepackage{svninfo}
\usepackage{prettyref}
\ifx\FrenchText\undefined
\usepackage[british]{babel}
\else
\usepackage[french]{babel}
\AtBeginDocument{%
\catcode`\:=12%
\catcode`\!=12%
}
\fi

\makeatletter

\ifx\Presentation\undefined
\fi

\newcommand{\fref}[1]{\prettyref{#1}}
\newrefformat{item}{\ref{#1}}
\newrefformat{apx}{Appendix~\ref{#1}}
\newrefformat{cha}{Chapter~\ref{#1}}
\newrefformat{sec}{Section~\ref{#1}}

\ifx\thmnum\undefined
\newtheorem{thmnum}{}[section]
\fi

\newcommand{\mynewthm}[3][]{%
  \def\PARAM{#1}
  \ifx\PARAM\empty
  \newtheorem{#2}[thmnum]{#3}
  \else
  \newtheorem{#2}{#3}[#1]
  \fi
  \newtheorem*{#2*}{#3}%
  \newrefformat{#2}{#3~\ref{##1}}%
}

\ifx\FrenchText\undefined
\newcommand{\ThmLabel}{Theorem}
\newcommand{\PrpLabel}{Proposition}
\newcommand{\LemLabel}{Lemma}
\newcommand{\FctLabel}{Fact}
\newcommand{\CorLabel}{Corollary}
\newcommand{\DfnLabel}{Definition}
\newcommand{\ConvLabel}{Convention}
\newcommand{\NtnLabel}{Notation}
\newcommand{\CstLabel}{Construction}
\newcommand{\ExmLabel}{Example}
\newcommand{\RmkLabel}{Remark}
\newcommand{\QstLabel}{Question}

\else
\newcommand{\ThmLabel}{\iflanguage{french}{Théorème}{Theorem}}
\newcommand{\PrpLabel}{Proposition}
\newcommand{\LemLabel}{\iflanguage{french}{Lemme}{Lemma}}
\newcommand{\FctLabel}{\iflanguage{french}{Fait}{Fact}}
\newcommand{\CorLabel}{\iflanguage{french}{Corollaire}{Corollary}}
\newcommand{\DfnLabel}{\iflanguage{french}{Définition}{Definition}}
\newcommand{\ConvLabel}{Convention}
\newcommand{\NtnLabel}{Notation}
\newcommand{\CstLabel}{Construction}
\newcommand{\ExmLabel}{\iflanguage{french}{Exemple}{Example}}
\newcommand{\RmkLabel}{\iflanguage{french}{Remarque}{Remark}}
\newcommand{\QstLabel}{Question}

\fi

\theoremstyle{plain}
\mynewthm{thm}{\ThmLabel}
\mynewthm{prp}{\PrpLabel}
\mynewthm{lem}{\LemLabel}
\mynewthm{fct}{\FctLabel}
\mynewthm{cor}{\CorLabel}

\theoremstyle{definition}
\mynewthm{dfn}{\DfnLabel}
\mynewthm{conv}{\ConvLabel}
\mynewthm{conj}{Conjecture}
\mynewthm{ntn}{\NtnLabel}
\mynewthm{cst}{\CstLabel}
\mynewthm{rstn}{Restriction}

\theoremstyle{remark}
\mynewthm{rmk}{\RmkLabel}
\mynewthm{qst}{\QstLabel}
\mynewthm{exm}{\ExmLabel}

\ifx\Presentation\undefined

\newcommand{\myenumlabel}[1]{\textnormal{(\roman{#1})}}
\renewcommand{\theenumi}{\myenumlabel{enumi}}
\fi

\newcounter{cycprfcnt}
\newenvironment{cycprf}%
{\begin{list}{\PackageWarning{begnac}{Label required for cycprf}}%
  {%
    \setcounter{cycprfcnt}{1}
    \setlength{\itemindent}{0.5\leftmargin}%
    \setlength{\leftmargin}{0pt}%
    \newcommand{\cpcurr}{\myenumlabel{cycprfcnt}}%
    \newcommand{\cpnext}{\addtocounter{cycprfcnt}{1}\cpcurr}%
    \newcommand{\cpprev}{\addtocounter{cycprfcnt}{-1}\cpcurr}%
    \newcommand{\cpnum}[1]{\setcounter{cycprfcnt}{##1}\cpcurr}%
    \newcommand{\cpfirst}{\cpnum{1}}%
    \newcommand{\impnext}{\cpcurr{} $\Longrightarrow$ \cpnext.}%
    \newcommand{\impfirst}{\cpcurr{} $\Longrightarrow$ \cpfirst.}%
  }%
}%
{\qedhere\end{list}}%

\def\indsym#1#2{%
  \setbox0=\hbox{$\m@th#1x$}%
  \kern\wd0%
  \hbox to 0pt{\hss$\m@th#1\mid$\hbox to 0pt{$\m@th#1^{#2}$\hss}\hss}%
  \lower.9\ht0\hbox to 0pt{\hss$\m@th#1\smile$\hss}%
  \kern\wd0}
\newcommand{\ind}[1][]{\mathop{\mathpalette\indsym{#1}}}

\def\nindsym#1#2{%
  \setbox0=\hbox{$\m@th#1x$}%
  \kern\wd0%
  \hbox to 0pt{\hss$\m@th#1\not$\kern1.4\wd0\hss}
  \hbox to 0pt{\hss$\m@th#1\mid$\hbox to 0pt{$\m@th#1^{#2}$\hss}\hss}%
  \lower.9\ht0\hbox to 0pt{\hss$\m@th#1\smile$\hss}%
  \kern\wd0}
\newcommand{\nind}[1][]{\mathop{\mathpalette\nindsym{#1}}}

\def\dotminussym#1#2{%
  \setbox0=\hbox{$\m@th#1-$}%
  \kern.5\wd0%
  \hbox to 0pt{\hss\hbox{$\m@th#1-$}\hss}%
  \raise.6\ht0\hbox to 0pt{\hss$\m@th#1.$\hss}%
  \kern.5\wd0}
\newcommand{\dotminus}{\mathbin{\mathpalette\dotminussym{}}}

\renewcommand{\emptyset}{\varnothing}
\renewcommand{\setminus}{\smallsetminus}
\def\models{\vDash}

\newcommand{\rest}{{\restriction}}

\newcommand{\sfrac}[2]{\hbox{$\frac{#1}{#2}$}}

\DeclareMathOperator{\tp}{tp}

\newcommand{\Cb}{\mathrm{Cb}}
\DeclareMathOperator{\Th}{Th}

\DeclareMathOperator{\dcl}{dcl}

\DeclareMathOperator{\id}{id}

\DeclareMathOperator{\Aut}{Aut}

\newcommand{\cA}{\mathcal{A}}
\newcommand{\cB}{\mathcal{B}}

\newcommand{\cT}{\mathcal{T}}
\newcommand{\cU}{\mathcal{U}}

\newcommand{\bE}{\mathbb{E}}

\newcommand{\bP}{\mathbb{P}}
\newcommand{\bQ}{\mathbb{Q}}
\newcommand{\bR}{\mathbb{R}}
\newcommand{\bS}{\mathbb{S}}

\makeatother

\author{Itaï Ben Yaacov}
\address{Itaï \textsc{Ben Yaacov} \\
  Université Claude Bernard -- Lyon 1 \\
  Institut Camille Jordan \\
  43 boulevard du 11 novembre 1918 \\
  69622 Villeurbanne Cedex \\
  France}

\urladdr{\url{http://math.univ-lyon1.fr/~begnac/}}

\author{Alexander Berenstein}
\address{Alexander Berenstein, Universidad de los Andes \\
  Cra 1 No 18A-10 \\
  Bogotá, Colombia \\
  and
  Université Claude Bernard -- Lyon 1 \\
  Institut Camille Jordan \\
  43 boulevard du 11 novembre 1918 \\
  69622 Villeurbanne Cedex \\
  France}
\urladdr{\url{http://matematicas.uniandes.edu.co/~aberenst}}

\author{C. Ward Henson}
\address{C. Ward Henson \\
  University of Illinois at Urbana-Champaign \\
  Urbana, Illinois 61801 \\
  USA}
\urladdr{\url{http://www.math.uiuc.edu/~henson}}

\thanks{The first author was supported by
  CNRS-UIUC exchange programme.
  The first and second authors were supported by
  ANR chaire d'excellence junior THEMODMET (ANR-06-CEXC-007).
  The third author was
  supported by NSF grants DMS-0100979, DMS-0140677 and DMS-0555904.}

\svnInfo $Id: Lp.tex 963 2009-07-28 07:02:29Z begnac $

\date{\today}

\DeclareMathOperator{\dist}{dist}

\title[Independence in $L_p(\mu)$]{Model-theoretic independence\\
in the Banach lattices $L_p(\mu)$}

\begin{document}

\begin{abstract}
We study model-theoretic stability and independence in Banach
lattices of the form $L_p(X,U,\mu)$, where $1 \leq p < \infty$. We
characterize non-dividing using concepts from analysis and show that
canonical bases exist as tuples of real elements.
\end{abstract}

\maketitle

\section{Introduction}

Let $X$ be a set, $U$ a $\sigma$-algebra on $X$ and $\mu$ a
measure on $U$, and let $p \in [1,\infty)$. We denote by
$L_p(X,U,\mu)$ the space of (equivalence classes of)
$U$-measurable functions $f\colon X \to \mathbb{R}$ such that
$\|f\|=(\int |f|^pd\mu)^{1/p}<\infty$.  We consider this space as
a Banach lattice (complete normed vector lattice) over $\bR$ in the
usual way; in particular, the lattice operations $\wedge,\vee$ are
given by pointwise maximum and minimum.  In this paper we study
model-theoretic stability and independence in such $L_p$ Banach
lattices.

There are several ways to understand model-theoretic stability for
classes of structures, like these, that lie outside the first order
context.  For the structures considered in this paper, these
approaches are completely equivalent. The work of Iovino
\cite{Iovino:StableBanach} provides tools for understanding stability in normed
space structures using the language of positive bounded formulas
developed by Henson \cite{Henson:NonstandardHulls}. (See also \cite{Henson-Iovino:Ultraproducts}.) A different
approach was initiated by Ben Yaacov \cite{BenYaacov:SimplicityInCats} in the compact
abstract theory (cat) setting \cite{BenYaacov:PositiveModelTheoryAndCats}; first order structures and
normed space structures are special cases.  Roughly speaking,
stability as developed in \cite{Iovino:StableBanach} and \cite{BenYaacov:SimplicityInCats} corresponds to
the study of universal domains in which there are bounds on the size
of spaces of types.  Buechler and Lessmann \cite{Buechler-Lessmann:SimpleHomogeneous} developed a
notion of simplicity (and thus also of stability) for strongly
homogeneous structures.  More recently, Ben Yaacov and Usvyatsov
\cite{BenYaacov-Usvyatsov:CFO} developed local stability for metric
structures in a continuous version of first order logic.  (See also
\cite{BenYaacov-Berenstein-Henson-Usvyatsov:NewtonMS}.)

In all four settings, a key ingredient is the analysis of a
model-theoretic concept of independence.  In \cite[part II, section
3]{Iovino:StableBanach} this analysis is based on a notion of non-forking, which
is characterized there using definability of types.  Independence is
studied in \cite{BenYaacov:SimplicityInCats} and in \cite{Buechler-Lessmann:SimpleHomogeneous} by means of the notion of
non-dividing (as defined by Shelah); in \cite[section 2]{BenYaacov:SimplicityInCats},
non-dividing in a stable structure is also characterized via
definability of types. In \cite{BenYaacov-Usvyatsov:CFO} the local
stability of continuous formulas is developed and a treatment of
independence is sketched in
\cite{BenYaacov-Berenstein-Henson-Usvyatsov:NewtonMS}. It follows
from what is proved in those papers that stability and independence
are the same notions from all four points of view, for the
structures to which they all apply.

In particular, for the structures studied here, these four
approaches to stability and independence are equivalent.  In this
paper, we take the concept of non-dividing as the foundation of our
study of independence.

We prove here that for each $p$ with $1 \leq p < \infty$, the Banach
lattices $L_p(\mu)$ are model-theoretically stable as normed space
structures, and we give a characterization of non-dividing using
concepts from analysis.  See
\cite{BenYaacov-Berenstein-Henson-Usvyatsov:NewtonMS} for a summary
of how these results can be translated into the setting of
continuous first order logic for metric structures.

Krivine and Maurey \cite{Krivine-Maurey:EspacesDeBanachStables} noted
that $L_p(\mu)$ spaces are stable, in a sense
that amounts to stability for \emph{quantifier-free} positive
formulas in the language of Banach \emph{spaces}. This observation
was part of a larger project, in which the main theorem is a deep
subspace property of quantifier-free stable (infinite dimensional)
Banach spaces: if $X$ is such a Banach space, then for some $p$ in
the interval $1 \leq p < \infty$, the sequence space $\ell_p$ embeds
almost isometrically in $X$.

We strengthen this stability observation about the $L_p(\mu)$
spaces so that it applies to the Banach \emph{lattice} setting and
to arbitrary positive bounded formulas in that language
(\emph{i.e.}, with bounded quantifiers allowed).  Our general
motivation is model-theoretic, in that we study independence,
non-dividing, and canonical bases in these structures.  In this
paper we do not attempt to derive structural results that apply to
all Banach lattices that are stable in the (strong) sense
considered here.

Our work is organized as follows.

In section 2, we introduce basic notions from analysis and
probability, such as conditional expectations and distributions.

In section 3 we recall model-theoretic results about atomless $L_p$
Banach lattices, concerning properties such as elimination of
quantifiers \cite{Henson-Iovino:Ultraproducts} and separable categoricity \cite{Henson:NonstandardHulls}, and we
review a characterization of types in terms of conditional
distributions that is due to Markus Pomper \cite{Pomper:PhD}. We prove that
the Banach lattices $L_p(\mu)$ are $\omega$-stable (with
respect to the metrics on the spaces of types that are induced by
the norm).
(This fact had been observed by the third
author \cite{Henson:SeparableBanachSpaces} but not published.) All of this is done in the
model theoretic context described in \cite{Henson-Iovino:Ultraproducts}, which develops the
language of \emph{positive bounded formulas} with an approximate
semantics.

In section 4 we use conditional expectations to characterize
non-dividing for the theory of atomless $L_p$ Banach lattices, thus
providing a relation between model-theoretic independence and the
notion of independence used in analysis and probability.

In section 5 we give a close analysis of the space of 1-types over a
given set of parameters in atomless $L_p$ Banach lattices.  In
particular, we give an explicit formula for calculating the distance
metric on that space of 1-types (\fref{cor:TypeDist}.)  This metric
is centrally important in the model theory of structures such as the
ones considered here.  The tools developed in this section also
yield a second characterization of model-theoretic independence in
these structures  (\fref{prp:independence-via-slices}.)

Finally, in section 6, we construct canonical bases for types in
atomless $L_p$ Banach lattices using conditional slices.
In particular we prove they always exist as sets of ordinary elements,
i.e., without any need for imaginary sorts.

\section{Basic analysis and probability}

We start with a review of some results from analysis that we use
throughout this paper.

Let $(X,U,\mu)$ be a measure space; that is, $X$ is a nonempty set,
$U$ is a $\sigma$-algebra of subsets of $X$, and $\mu$ is a
$\sigma$-additive measure on $U$ (not necessarily assumed to be
finite or even $\sigma$-finite, in general).

A measure space $(X,U,\mu)$ is called \emph{decomposable} (also
called \emph{strictly localizable}) if there exists a partition
$\{X_i: i \in I\}\subset U$ of $X$ into measurable sets such that
$\mu(X_i)<\infty$ for all $i \in I$ and such that for any subset $A$
of $X$, $A \in U$ iff $A \cap X_i \in U$ for all $i \in I$ and, in
that case, $\mu(A)=\sum_{i \in I}\mu(A\cap X_i)$. When these
properties hold, the partition $\{X_i: i \in I\}$ will be called a
\emph{witness} for the decomposability of $(X,U,\mu)$.
(See \cite{Haydon-Levy-Raynaud:RandomlyNormedSpaces}.)

\begin{conv}\label{decomposable}
Throughout this paper we require that all measure spaces are
decomposable.
\end{conv}

If $(X,U,\mu)$ and $(Y,V,\nu)$ are measure spaces, we define a
product measure space $(X,U,\mu) \otimes (Y,V,\nu)$ by defining $\mu
\otimes \nu$ in the usual way on rectangles $A \times B$, where $A
\in U$ and $B \in V$ have finite measure, and then extending to make
the resulting product measure space on $X \times Y$ decomposable.
When both measure spaces are $\sigma$-finite, this agrees with the
usual product measure construction.

Let $1 \leq p < \infty.$ A Banach lattice $E$ is an \emph{abstract
$L_p$-space} if $\| x+y \| ^p = \| x \|^p + \| y \|^p$ whenever $x,y
\in E$ and $x \wedge y = 0$.  Evidently $L_p(X,U,\mu)$ is an abstract
$L_p$-space for every measure space $(X,U,\mu)$. For the study of
$L_p(\mu)$ spaces, the requirement that all measure spaces be
decomposable causes no loss of generality; indeed, the
representation theorem for abstract $L_p$-spaces states that each
such space is $L_p(X,U,\mu)$ for some decomposable measure space
$(X,U,\mu)$.  (Discussions of this representation theorem can be
found in \cite[pp.\ 15--16]{Lindenstrauss-Tzafriri:ClassicalBanachSpacesII} and \cite[Chapter 5]{Lacey:IsometricTheory}; see
\cite[p.\ 135]{Lacey:IsometricTheory} for the history of this result. See also the
proof of Theorem 3 in \cite{Bretagnolle-DacunhaCastelle-Krivine:LoisStables}, which was a key paper in the model
theory of $L_p$-spaces.)

Let $E$ be any Banach lattice and $f\in E$.  The \emph{positive
part} of $f$ is $f \vee 0$, and it is denoted $f^+$.  The
\emph{negative part} of $f$ is $f^- = (-f)^+$, and one has $f = f^+
- f^-$ and $|f| = f^++f^-$.  Further, $f$ is \emph{positive} if
$f=f^+$ and $f$ is \emph{negative} if $-f$ is positive. For $f,g\in
E$, one has $f\geq g$ iff $f-g$ is positive.

A subspace $F\subset E$ is an \emph{ideal} if whenever $g\in E$ and
$f\in F$ are such that $0\leq |g|\leq |f|$, one always has $g\in F$.
An ideal $F$ is a \emph{band} if for all collections $\{h_j: j \in
J\}\subset F$ such that $h=\bigvee_{j\in J} h_j\in E$, one always
has $h\in F$. By a \emph{sublattice} of $E$ we mean a norm-closed
linear sublattice.  If $A,B$ are subsets of $E$ we write $A \leq B$
to mean that $A$ and $B$ are sublattices of $E$ and $A$ is contained
in $B$.

Let $B\subset E$. One defines $B^\perp=\{f \in E \colon
|f|\wedge|g|=0$ for all $ g\in B\}$ and therefore $B^{\perp
\perp}=\{f \in E \colon |f|\wedge |g|=0 \text{ for all } g\in
B^\perp\}$. It is a standard fact that $B^\perp$ and $B^{\perp
\perp}$ are bands and $B^{\perp \perp}$ is the smallest band
containing $B$.  One refers to $B^\perp$ as the \emph{band
orthogonal to} $B$ and to $B^{\perp \perp}$ as the \emph{band
generated by} $B$.  In $L_p$-spaces, every band is a projection
band.  That is, for any set $B \subset L_p(X,U,\mu)$, one has a
lattice direct sum decomposition $L_p(X,U,\mu) = B^\perp \oplus
B^{\perp\perp}$. (See \cite{Schaefer:BanachLattices}, for example.)

Let $(X,U,\mu)$ be a measure space. A measurable set $S\in U$ is an
\emph{atom} if $\mu(S) > 0$ but there do not exist $S_1,S_2\in U$
disjoint, both of positive measure, such that $S_1\cup S_2=S$. One
calls $(X,U,\mu)$ \emph{atomless} if it has no atoms.  An atom in a
Banach lattice is an element $x$ such that the ideal generated by
$x$ has dimension 1.  In $L_p(X,U,\mu)$, the atoms in the sense of
this definition are exactly the elements of the form $r\chi_S$ where
$r \neq 0$ and $S$ is an atom in the sense of measure theory. We may
write $X$ as the disjoint union of two measurable sets, $X_0$ and
$X_1$, such that $X_0$ is (up to null sets) the union of all atoms
in $U$ and $X_1$ is atomless. Moreover, if $B$ is the set of atoms
in $L_p(X,U,\mu)$, then $B^{\perp\perp}=L_p(X_0,U_0,\mu)$ and
$B^\perp = L_p(X_1,U_1,\mu)$, where for each $i=0,1$, $U_i$ is the
restriction of $U$ to $X_i$.

\begin{dfn}
Let $(X,U,\nu)$ and $(Y,V,\mu)$ be a measure spaces. We write
$(Y,V,\nu)\subset (X,U,\mu)$ to mean that $V \subset U$ and $\nu(A)
= \mu(A)$ for all $A \in V$, and that there exists a witness $(X_i
\mid i \in I)$ for the decomposability of $(X,U,\mu)$ and $J \subset
I$ such that $(X_i \mid i \in J)$ witnesses the decomposability of
$(Y,V,\nu)$.  In particular $Y$ and each of the sets in $(X_i \mid i
\in J)$ are elements of $V$.
\end{dfn}

\begin{ntn}
In the rest of this paper, we will frequently use $\mu$ as the generic
symbol for a measure.  This follows usual mathematical practice, as when
$+$ is used as the symbol for addition in every abelian group.  In particular,
when we write $(Y,V,\mu) \subset (X,U,\mu)$, it is the restriction of $\mu$
to $V$ that is to be used as the measure in the measure space $(Y,V,\mu)$.
\end{ntn}

\begin{rmk}\label{representation}  Let $\cU$ be an abstract
$L_p$-space and let $(C_j \colon j=1,\dots,n)$ be an increasing
chain of  sublattices of $\cU$, so $C_1  \leq  \dots  \leq C_n \leq
\cU$.  Note that each $C_j$ is an abstract $L_p$ space. One can use
the representation theorem for abstract $L_p$ spaces to show that
there exist measure spaces $(X,U,\mu)$ and $((Y_j,V_j,\mu) \colon
j=1,\dots,n)$ satisfying $(Y_1,V_1,\mu) \subset \dots \subset
(Y_n,V_n,\mu) \subset (X,U,\mu)$, as well as an isomorphism $\Phi$
from $L_p(X,U,\mu)$ onto $\cU$ such that $\Phi$ maps
$L_p(Y_j,V_j,\mu)$ exactly onto $C_j$ for each $j$.  To see this,
proceed inductively to construct for each $j$ a maximal set $S_j$ of
pairwise disjoint positive elements of $C_j$, satisfying $S_1
\subset \dots \subset S_n$, and extend $S_n$ to a maximal set $T$ of
pairwise disjoint positive elements of $\cU$. Then the elements of $U$
(respectively, $V_j$) having finite measure can be identified with
the set of elements of $\cU$ (respectively, of $C_j$) that are
convergent sums of disjoint components of elements of $T$
(respectively, of $S_j$), and the measure $\mu$ is simply $\|\
\|^p$. The full measure spaces are then determined by taking them to
be decomposable.
\end{rmk}

A key relationship between an abstract $L_p$-space $\cU$ and a
given sublattice $C$ of $\cU$ is based on the following standard
theorem.  As explained in Remark \ref{representation}, we may take
$\cU = L_p(X,U,\mu)$ and $C=L_p(Y,V,\mu)$ where $(Y,V,\mu)
\subset (X,U,\mu)$.

\begin{fct}(Conditional Expectation)\label{fct:CondExp}
Let $(Y,V,\mu)\subset (X,U,\mu)$ be measure spaces (and recall
Convention \ref{decomposable}, which applies to both $(X,U,\mu)$
and $(Y,V,\mu))$.  Fix $1 \leq p < \infty$ and let $f \in
L_p(X,U,\mu)$.  Then there exists a unique $g_f \in L_p(Y,V,\mu)$
such that $\int_{A} g_f \, d\mu = \int_{A}f \, d\mu$ for every
$A\in V$. We call the function $g_f$ \emph{the conditional
expectation of $f$ with respect to $(Y,V,\mu)$} and denote it by
$\bE(f|V)$.  If $f = \chi_B$ for some $B \in U$, we  often
write $\bE(B|V)$ in place of $\bE(f|V)$.
The operator mapping $f$
to $\bE(f|V)$ is a contractive, positive projection from
$L_p(X,U,\mu)$ onto $L_p(Y,V,\mu)$, and $\bE(f|V)=0$ for any $f \in
L_p(Y,V,\mu)^\perp$.
\end{fct}

\begin{proof} If $\mu(Y) < \infty$, the existence of $g_f$ is a
standard fact; first restrict $f$ and $U$ to $Y$, and then apply
the usual conditional expectation operator.  Our assumption that
$(Y,V,\mu)$ is decomposable gives the general case.  See the proof
of \cite[Theorem 11.4, p.\ 212]{Schaefer:BanachLattices} for details.
A somewhat more
elementary proof of the existence of a contractive, positive
projection from $L_p(X,U,\mu)$ onto $L_p(Y,V,\mu)$ is given in
\cite{Lindenstrauss-Tzafriri:ClassicalBanachSpacesII}. (See Lemma 1.b.9 and its proof on page 20.)
\end{proof}

\begin{rmk}\label{l2projection}
Consider the case $p=2$; $L_2(X,U,\mu)$ is the expansion of a
Hilbert space by adding the lattice operations $\wedge$, $\vee$.
Let $(Y,V,\mu) \subset (X,U,\mu)$ and let $f \in L_2(X,U,\mu)$.
From the definitions, one has $\langle f -\bE(f|V),\chi_A \rangle=0$ for
all $A \in V$; using linearity and continuity of the inner product
as well as density of simple functions, it follows that $\langle f
-\bE(f|V),h \rangle=0$ for all $h\in L_2(Y,V,\mu)$. Thus $\bE(f|V)$ is
the orthogonal projection of $f$ on the closed subspace $L_2(Y,V,\mu)$.
\end{rmk}

For our characterization of model-theoretic independence, it is
important that the conditional expectation operator defined above is
uniquely determined by some of its Banach lattice properties. This
is shown in the following result using functional analysis; a
second, more elementary proof is given in \fref{sec:CondSlice}.
(See
\fref{prp:CondSliceExpectNorm} and \fref{rmk:SecondCondExpUniqueness}).

\begin{prp}\label{prp:CExpUnq}
  Let $\cU$ be an abstract $L_p$ Banach lattice and let $C$ be any
  sublattice of\ \ $\cU$.  There is a unique linear operator $T
  \colon \cU \to C$ such that $T$ is a contractive, positive projection
  and $T(f)=0$ for any $f \in C^\perp$.  Indeed, if $(Y,V,\mu)
  \subset(X,U,\mu)$ are any measure spaces and $\Phi$ is any
  isomorphism from $L_p(X,U,\mu)$ onto $\cU$ that maps $L_p(Y,V,\mu)$
  exactly onto $C$, then for any $f \in \cU$ one has that $T(f)$ =
  $\Phi^{-1}(\bE(\Phi(f)|V))$.
\end{prp}

\begin{proof} The existence of such an operator and its connection
  to the conditional expectation is given in Fact \ref{fct:CondExp}.

  It remains to prove uniqueness.  Suppose $T \colon \cU \to C$ is a
  contractive, positive projection and $T(f)=0$ for any $f \in
  C^\perp$.  We first show that if $f \in C$, then $T$ maps the band
  $B=\{f\}^{\perp\perp}$ generated by $f$ into itself. To see this,
  note that if $0 \leq x \leq |f|$, then positivity of $T$ implies $0
  \leq T(x) \leq T(|f|) = |f|$, so $T(x) \in B$.  In $L_p$-spaces, $B$
  is the closed linear span of such $x$, so $T$ necessarily maps $B$
  into itself.

  Let $(Y,V,\mu) \subset(X,U,\mu)$ be measure spaces such that $\cU =
  L_p(X,U,\mu)$ and $C = L_p(Y,V,\mu)$.  Recalling that $(Y,V,\mu)$ is
  decomposable and using the band argument in the second paragraph of
  this proof, we may reduce to the case where $Y=X$ and $\mu(X) <
  \infty$. Without loss of generality we take $\mu(X)=1$.

  Let $q$ be dual to $p$, so $q = \infty$ if $p=1$ and $p+q = pq$
  otherwise.  Then $L_q(X,U,\mu)$ is the dual space of
  $L_p(X,U,\mu)$, with the pairing given by $\langle f,g \rangle = \int_X
  fg\, d\mu $ for all $f \in L_p(X,U,\mu)$ and $g \in L_q(X,U,\mu)$.
  Let $T' \colon L_q(X,V,\mu) \to L_q(X,U,\mu)$ be the adjoint of
  $T$, a positive linear operator of norm $1$.  Note that $T'(\chi_X) =
  \chi_X$, since $T'(\chi_X)$ is positive, has norm $\leq 1$, and
  satisfies $\langle \chi_X, T'(\chi_X) \rangle = \langle T(\chi_X),\chi_X
  \rangle = \langle \chi_X,\chi_X \rangle = \mu(X) = 1$.

  To prove $T$ is unique, it suffices to prove $T(f) = \bE(f|V)$ for
  any $f \in L_p(X,U,\mu)$ that satisfies $0 \leq f \leq \chi_X$,
  since the linear span of such functions is norm dense in
  $L_p(X,U,\mu)$. Let $A \in V$ and set $B = X \setminus A \in V$
  so $f = f\chi_A + f\chi_B$.
  The band argument above shows that $T(f\chi_A)$ vanishes on $B$
  and that $T(f\chi_B)$ vanishes on $A$.  Therefore we have
  \[ \int_A T(f)\, d\mu = \int_X \chi_AT(f)\, d\mu = \int_X
  \chi_AT(f\chi_A)\, d\mu + \int_X \chi_AT(f\chi_B)\, d\mu \]
  \[ = \int_X \chi_A T(f\chi_A)\, d\mu = \int_X \chi_X T(f\chi_A)\, d\mu =
  \int_X T'(\chi_X)f \chi_A\, d\mu \]
  \[ = \int_X \chi_X f \chi_A\, d\mu = \int_X \chi_A f \, d\mu = \int_A
  f\, d\mu. \]
  Hence $T(f)$ must be the conditional expectation of $f$ relative
  to $(X,V,\mu)$.
\end{proof}

\begin{rmk}
  The uniqueness of $T$ could have been derived from a theorem of
  Douglas \cite{Douglas:ContractiveProjections} (for $p=1$) and Ando
  \cite{Ando:ContractiveProjections} (for $p>1$).
  The
  restriction of $T$ to a band generated by a single element of $C$
  corresponds to the case of \ref{prp:CExpUnq} where $Y=X$ and
  $\mu(X)<\infty$, which is exactly the setting of the Douglas-Ando
  result.  Applying this to each such band would give the uniqueness
  of $T$ globally. In order to make this paper more self-contained and
  because we did not find in the literature a simple proof of the
  Douglas-Ando result for all values of $p$, we included one here. It
  is adapted from the argument for the $p=1$ case in
  \cite{Abramovich-Aliprantis-Burkinshaw:ContractiveProjections}.
\end{rmk}

\begin{rmk} In Proposition \ref{prp:CExpUnq} the assumption that
  $T=0$ on $C^\perp$ is needed for $p=1$ but it follows from the
  other assumptions about $T$ when $p>1$.
\end{rmk}

While the definition of conditional expectation is in terms of
functions on concrete measure spaces, it follows from the moreover
part of \fref{prp:CExpUnq} that the conditional
expectation $\bE(f|V)$ only depends on the embedding of $C$ in $\cU$
as an abstract sublattice.

\begin{ntn}\label{ntn:AbstractCondExp}
  If $\cU$ is an abstract $L_p$ lattice and $C$ is a
  sublattice of $\cU$ then the conditional expectation mapping from
  $\cU$ to $C$ will be denoted by $\bE_C^\cU$,
  or simply by $\bE_C$ if $\cU$ is understood from the context.
\end{ntn}

\begin{dfn} Let $(Y,V,\mu)\subset (X,U,\mu)$.
Let $\bar f=(f_1,...,f_n) \in
L_p(X,U,\mu)^n$ and let $\bar g=(g_1,...,g_n) \in L_p(X,U,\mu)^n$ be
such that $f_i$, $g_i$ are in the band generated by $L_p(Y,V,\mu)$
for $i\leq n$. We write $\dist(f_1,...,f_n|V)=\dist(g_1,...,g_n|V)$
and say that $(f_1,...,f_n)$, $(g_1,...,g_n)$ have the same
\emph{(joint) conditional distribution over} $(Y,V,\mu)$ if
\centerline{$\bE(\bar f^{-1}(B)|V)=\bE(\bar g^{-1}(B)|V)$} for
any Borel set $B\subset \mathbb{R}^n$.
\end{dfn}

\begin{dfn}
Let $(X,V,\mu)\subset (X,U,\mu)$. The measure space $(X,U,\mu)$ is
\emph{atomless over} $(X,V,\mu)$ if for every $A\in U$ of positive
finite measure there exists $B \in U$ such that $A \cap B \neq A
\cap C$ for all $C \in V$.
\end{dfn}

A key property of measure spaces having this ``atomless over''
relation is the following result
\cite[Theorem~1.3]{Berkes-Rosenthal:AlmostExchangeableSequences}.
See also
\cite[Lemma~331B]{Fremlin:MeasureTheoryVol3}.

\begin{fct}[Maharam's Lemma]\label{maharam}
Let $(X,V,\mu)\subset (X,U,\mu)$ with $(X,U,\mu)$ atomless over
$(X,V,\mu)$; then for every $A\in U$ of positive finite measure
and for every $f\in L_p(X,V,\mu)$ such that $0 \leq f \leq
\bE(A|V)$ there is a set $B\in U$ such that $B\subset A$ and
$\bE(B|V)=f$.
\end{fct}

\begin{exm}\label{lebesgue}
Let $(Y,V,\mu)$ be a measure space and let $([0,1],\cB,m)$ be the
standard Lebesgue measure space on the interval $[0,1]$. Let $\cT$
be the trivial $\sigma$-algebra on $[0,1]$: $\cT=\{\emptyset,[0,1]\}$.
Then $(Y,V,\mu)\otimes ([0,1],\cB,m)$ is atomless over
$(Y,V,\mu)\otimes ([0,1],\cT,m)$, which is isomorphic to
$(Y,V,\mu)$.
\end{exm}

The following result shows how to obtain functions with a
prescribed conditional distribution. We state the result for $L_p$
functions; in \cite{Berkes-Rosenthal:AlmostExchangeableSequences} Berkes and Rosenthal give a proof of the
corresponding result for arbitrary measurable functions.

\begin{fct}\label{BR}(Theorem 1.5 \cite{Berkes-Rosenthal:AlmostExchangeableSequences})
Let $(X,V,\mu)\subset (X,U,\mu)$ be measure spaces, where
$(X,U,\mu)$ is atomless over $(X,V,\mu)$. Let $(X,V,\mu) \subset
(X,W,\mu)$ be any other extension, not necessarily atomless over
$(X,V,\mu)$. Then for any $f\in L_p(X,W,\mu)$, there is $g\in
L_p(X,U,\mu)$ such that $\dist(f|V)=\dist(g|V)$.
\end{fct}

We need an iterated version of the previous result, as stated
in the next lemma. This was proved by Markus Pomper in his thesis
\cite[Theorem 6.2.7]{Pomper:PhD}.  Pomper gave a direct proof based
on generalizing the argument of Berkes and Rosenthal to dimension
$>1$.  We give a different proof by iterating the $1$-dimensional
result.

\begin{lem}\label{BR-P}
Let $(X,V,\mu)\subset (X,U,\mu)$ be measure spaces, where
$(X,U,\mu)$ is atomless over $(X,V,\mu)$. Let $(X,V,\mu) \subset
(X,W,\mu)$ be any other extension, not necessarily atomless over
$(X,V,\mu)$.  Then for any $f_1,...,f_n\in L_p(X,W,\mu)$, there
are $g_1,...,g_n\in L_p(X,U,\mu)$ such that
$\dist(f_1,...,f_n|V)=\dist(g_1,...,g_n|V)$.
\end{lem}

\begin{proof}  We reduce this Lemma to the case in which
$(X,U,\mu)$ is isomorphic over $(X,V,\mu)$ to $(X,V,\mu) \otimes
([0,1],\cB,m)$, where $([0,1],\cB,m)$ is the standard Lebesgue
space.  To see that a treatment of this special case is
sufficient, it suffices to show that any $(X,U,\mu)$ that is
atomless over $(X,V,\mu)$ must contain a measure space
isomorphic over $(X,V,\mu)$ to $(X,V,\mu) \otimes ([0,1],\cB,m)$.
To prove this, it suffices to consider the case where $\mu(X) <
\infty$ since $(X,U,\mu)$ is decomposable. Fact \ref{maharam}
yields $X_0 \in U$ with $(\chi_X)/2 = \bE(X_0|V)$. Setting $X_1 = X
\setminus X_0$ we get a partition $X_0,X_1 \in U$ of $X$ such that
$(\chi_X)/2 = \bE(X_0|V) = \bE(X_1|V)$.  Now Fact \ref{maharam} can
be applied to each of the sets $X_0$ and $X_1$, yielding a
refinement of $X_0,X_1$ to a partition of $X$ by four sets
$X_{00},X_{01},X_{10},X_{11}$ from $U$, each of which satisfies
$(\chi_X)/4 = \bE(X_{ij}|V)$. Continuing to apply Fact
\ref{maharam} inductively yields a family $(X_\alpha)$ in $U$,
where $\alpha$ ranges over all finite sequences from $\{0,1\}$,
which has the following properties: (i) for each $n \geq 1$,
$\pi_n = \{X_\alpha\colon\alpha \mbox{ has length } n\}$ is a
partition of $X$; (ii) $\pi_{n+1}$ refines $\pi_n$ for each $n
\geq 1$; and (iii) if $\alpha$ has length $n$, then
$\bE(X_\alpha|V) = 2^{-n}\chi_X$.  The measure subspace of
$(X,U,\mu)$ generated by $V$ and the sets $X_\alpha$ is isomorphic
over $(X,V,\mu)$ to $(X,V,\mu) \otimes ([0,1],\cB,m)$.

So, we now assume that $(X,U,\mu)$ is isomorphic over $(X,V,\mu)$
to $(X,V,\mu) \otimes ([0,1],\cB,m)$. Note that $([0,1],\cB,m)$ is
isomorphic to $([0,1]^n,\cB_n,m_n)$, where $\cB_n$ is the Lebesgue
$\sigma$-algebra of the product space, and $m_n$ is the product
measure. Thus we may assume that $(X,U,\mu)$ is equal to
$(X,V,\mu)\otimes ([0,1]^n,\cB_n,m_n)$. Let $\cT_k$ be the trivial
$\sigma$-algebra on $[0,1]^k$: $\cT_k=\{\emptyset,[0,1]^k\}$. Let $V_i$
be the subalgebra $(X,V,\mu)\otimes ([0,1]^i,\cB_i,m_i) \otimes
([0,1]^{n-i},\cT_{n-i},m_{n-i})$ of $(X,W,\mu)\otimes
([0,1]^n,\cB_n,m_n)$ for each $i\leq n$. Then $(X,V_{i+1},\mu)$ is
atomless over $(X,V_{i},\mu)$ for all $i\leq n-1$.

Now we proceed inductively. By Fact \ref{BR}, there is $g_1\in
L_p(X,V_1,\mu)$ such that $\dist(f_1|V)=\dist(g_1|V)$. For the
induction step, suppose there are $g_1,...,g_l\in
L_p(X,V_{l},\mu)$ such that
$\dist(f_1,...,f_l|V)=\dist(g_1,...,g_l|V)$. Let $W_l$ be the
smallest $\sigma$-subalgebra of $W$ containing $V$ and making
$f_1,...,f_l$ measurable; let $U_l$ be the smallest
$\sigma$-subalgebra of $U$ containing $V$ and making $g_1,...,g_l$
measurable, so $U_l \subset V_l$. Then there is an isomorphism
from $(X,W_l,\mu)$ to $(X,U_l,\mu)$ which is equal to the identity
when restricted to $V$. Since $(X,V_{l+1},\mu)$ is atomless over
$(X,V_l,\mu)$, it is also atomless over $(X,U_l,\mu)$. By \ref{BR}
we can find $ g_{l+1}$ in $L_p(X,V_{l+1},\mu)$ such that
$\dist(f_1,...,f_{l+1}|V)=\dist(g_1,...,g_{l+1}|V)$.
\end{proof}

\begin{cor}\label{existence}
Let $(X,W,\mu)\subset (X,V,\mu)\subset (X,U,\mu)$ be measure spaces
such that $(X,U,\mu)$ is atomless over $(X,V,\mu)$. Let
$(X,W,\mu)\subset (X,V',\mu)$ be any other extension. Then for any
$\bar f=(f_1,\dots,f_n)\in L_p(X,V',\mu)$ there are
$\bar g=(g_1,\dots,g_n)\in L_p(X,U,\mu)$ such that
$\dist(g_1,\dots,g_n|W)=\dist(f_1,\dots,f_n|W)$ and
$\bE(\bar g^{-1}(B)|V)=\bE(\bar g^{-1}(B)|W)$ for any Borel set $B\subset
\mathbb{R}^n$.
\end{cor}

\begin{proof}
By Lemma \ref{BR-P} and Example \ref{lebesgue}, we may assume that
$(X,V',\mu)=(X,W,\mu)\otimes ([0,1],\cB,m)$, where $([0,1],\cB,m)$ is
the standard Lebesgue space. We view $(X,V,\mu)\otimes ([0,1],\cB,m)$
as an extension of $(X,V,\mu)$, and also as an extension of
$(X,V',\mu)$.  Hence we can regard $\bar f=(f_1,\dots,f_n)\in
L_p(X,V',\mu)$ as elements of the $L_p$-space of an extension of
$(X,V,\mu)$. By Lemma \ref{BR-P} there are $\bar g=(g_1,\dots,g_n)\in
L_p(X,U,\mu)$ with $\dist(g_1,\dots,g_n|V)=\dist(f_1,\dots,f_n|V)$.
Note also that for any Borel set $B\subset \mathbb{R}^n$, we have
$\bE(\bar g^{-1}(B)|V)=\bE(\bar f^{-1}(B)|V)=\bE(\bar f^{-1}(B)|W)$.

\end{proof}

\section{Basic model theory}

We use the model-theoretic tools developed in  \cite{Henson-Iovino:Ultraproducts} to study
normed space structures.  We use the word \emph{formula} to mean
\emph{positive bounded formula} and use the semantics of
\emph{approximate satisfaction} as defined there.  In particular,
the \emph{type} of a tuple over a set, denoted $\tp(\bar f,/A)$,
is the collection of such
formulas (with parameters from the set) approximately satisfied by
the tuple.
We also write $\bar f \equiv_C \bar g$ in order to say that
$\tp(\bar f/C) = \tp(\bar g/C)$.

We study the model theory of $L_p$ Banach lattices in the signature
$\mathcal{L}=\{0,-,(f_q\mid q \in \mathbb{Q}),+,\wedge,\vee,\| \ \
\|\}$, where each $f_q$ is interpreted as the unary function of
scalar multiplication by $q$. Terms in this signature correspond to
lattice polynomials; atomic formulas are of the forms $\|t\| \leq r$
and $s \leq \|t\|$ where $t$ is a term and $r,s$ are rational
numbers. By $\Th_\cA(L_p(X,U,\mu))$ we mean the approximate positive
bounded theory of the Banach lattice $L_p(X,U,\mu)$ in this
signature. See \cite{Henson-Iovino:Ultraproducts} for the necessary background.

\begin{rstn}\label{atomless}
In the rest of this paper we only consider the model theory of $L_p$ Banach lattices that are based on \emph{atomless} measure spaces.
\end{rstn}

Because of the direct sum decomposition of a measure space into its
atomless and purely atomic parts, it is routine to extend what is
done here to obtain analogous results in the general case.  See the
end of this section for a brief discussion.

\begin{fct}(See Theorem 2.2 in \cite{Henson:NonstandardHulls})
If $(X,U,\mu)$ and $(Y,V,\nu)$ are atomless measure spaces, then
their $L_p$ Banach lattices are elementarily equivalent;
\textit{i.e.}, \newline
\centerline{$\Th_\cA(L_p(X,U,\mu)) = \Th_\cA(L_p(Y,V,\nu))$.}
\end{fct}

\begin{fct}(Axiomatizability, see Example 13.4 in \cite{Henson-Iovino:Ultraproducts}
and Theorem 2.2 in \cite{Henson:NonstandardHulls}) Let $(X,U,\mu)$ be an atomless
measure space. Let $M$ be a Banach space structure such that
$M\models_\cA \Th_\cA(L_p(X,U,\mu))$. Then $M$ is an atomless
$L_p$ Banach lattice; \textit{i.e.}, there is an atomless measure
space $(Y,V,\nu)$ such that $M$ is isomorphic to $L_p(Y,V,\nu)$.
\end{fct}

\begin{dfn}
Let $\kappa$ be a cardinal larger than $2^{\aleph_0}$. We say that
a normed space structure $\cU$ is a \emph{$\kappa$-universal domain}
if it is $\kappa$-strongly homogeneous and $\kappa$-saturated. We
call a subset $C\subset \cU$ \emph{small} if $|C|<\kappa$.
\end{dfn}
Thus, if $\bar f$,$\bar g\in \cU^n$ are two tuples and $C \subseteq \cU$ is a set
(by which we always implicitly mean that $|C| < \kappa$):
$\bar f \equiv_C \bar g$ if and only if there exists an automorphism $\theta$ of
the Banach lattice $\cU$ which fixes $C$ (in symbols: $\theta \in \Aut(\cU/C)$)
sending $\bar f$ to $\bar g$.

For each cardinal $\kappa$ and each consistent set of positive
bounded sentences $\Sigma$, there exists a normed space structure
that is a model of $\Sigma$ and a $\kappa$-universal domain; see
\cite[Corollary 12.3 and Remark 12.4]{Henson-Iovino:Ultraproducts}.

For the remainder of this section, $\cU$ will be a $\kappa$-universal
domain for the theory of atomless $L_p$ Banach lattices, where
$\kappa$ is much larger than any set or collection of variables or
constants under consideration. Since $\cU$ is at least
$\omega_1$-saturated, $\cU$ is a metrically complete structure and
there is a measure space $(X,U,\mu)$ such that $\cU=L_p(X,U,\mu)$.
Unless stated otherwise, sets of parameters such as $A,B,C\subset
\cU$ are required to be small.

\begin{fct}\label{sepcat} (Separable categoricity)
Let $M\models_\cA \Th_\cA(\cU)$ be separable and complete. Then
$M$ is isomorphic to $L_p([0,1],\cB,m)$, where $\cB$ is the
$\sigma$-algebra of Lebesgue measurable sets and $m$ is Lebesgue
measure.
\end{fct}

Note that Fact \ref{sepcat} need not hold if we add constants to
the language. By Fact \ref{QE} below, we get
$(L_p([0,1],\cB,m),f))\equiv_\cA (L_p([0,1],\cB,m),g))$, where $f$
and $g$ are any two norm 1, positive elements of
$L_p([0,1],\cB,m)$. However, there are two possible isomorphism
types of such structures, depending on whether or not the support
of the adjoined function has measure 1 or not.

\begin{fct}\label{QE} (Quantifier elimination, see Example 13.18
in \cite{Henson-Iovino:Ultraproducts}) Let $\bar a,\bar b \subset \cU$. If $\bar a$ and $\bar b$ have the
same quantifier-free type in $\cU$, then $\bar a$ and $\bar b$ have the
same type in $\cU$. That is, if $\|t(\bar a)\|=\|t(\bar b)\|$ for every
term $t$, then $\bar a$ and $\bar b$ have the same type.
\end{fct}

Note that Fact \ref{QE} fails to be true without the assumption that
$\cU$ is atomless; atoms and non-atoms can have the same
quantifier-free type, but they never have the same type.

An important tool for studying non-dividing (which we do in the
next section) is a characterization of types in terms of
conditional distributions. The following results were proved by
Markus Pomper \cite[Theorems 6.3.1 and 6.4.1]{Pomper:PhD} in his thesis.
We give alternate proofs.

\begin{prp}
  \label{prp:TypeIsDist}
  Let $B$ be a  sublattice of\ \ $\cU$ and let
  $(Y,V,\mu)\subset (X,U,\mu)$ be measure spaces such that
  $B=L_p(Y,V,\mu)$ and $\cU=L_p(X,U,\mu)$.
  Let $\bar f,\bar h \in (B^{\perp \perp})^n$.
  Then the following are equivalent:
  \begin{enumerate}
  \item $\bar f\equiv_B \bar h$.
  \item $\sum \lambda_if_i \equiv_B \sum \lambda_ih_i$ for all $\bar \lambda \in \bR^n$.
  \item $\dist(\bar f|V)=\dist(\bar h|V)$.
  \end{enumerate}
\end{prp}
\begin{proof}
  \begin{cycprf}
  \item[\impnext] Immediate.
  \item[\impnext]
    The joint conditional
    distribution of a tuple is determined by the family of conditional
    distributions of linear combinations of the coordinates.
    (This is
    proved using characteristic functions or Laplace transforms of
    distributions; see \cite[Theorem 5.3]{Kallenberg:ModernProbability}.)
    Thus it will be enough to show that
    $f \equiv_B h$ implies that $\dist(f|V) = \dist(h|V)$ (where
    $f,h \in B^{\perp\perp}$).
    By splitting into
    positive and negative parts, we may assume that $f = f^+$.
    Moreover, since all measure spaces are decomposable here, there is
    no loss of generality in assuming $\mu(X) < \infty$.  Hence also
    $\mu(Y) < \infty$ and the characteristic function $\chi_Y$ is in
    $B$.

    Assume first that $f=\chi_A$ for some set $A\in U$; because $f \in
    B^{\perp\perp}$ we have $A \subset Y$ and hence $f$ satisfies $f \wedge (\chi_Y - f) = 0$.
This can be equivalently expressed by the family of conditions $f_C
\wedge (\chi_C - f_C) = 0$, where $C$ ranges over the sets of finite
measure in $V$ and $f_C$ denotes the restriction of $f$ to $C$,
which is $f \wedge \chi_C$ since $f$ is a characteristic function.
In particular, this condition on $f$ is expressible by a family of
conditions whose parameters come from $B$.
    Assuming $\tp(h/B)=\tp(f/B)$ and $h \in B^{\perp\perp}$, it follows that
    $h$ is the characteristic function of some subset of $Y$ and that
    $\|f\|=\|h\|$. It is further easy to verify that $\dist(f|V)=\dist(h|V)$.

    Now assume that $f$ is a simple positive function, written as
    $f=r_1\chi_{A_1}+\dots+r_n\chi_{A_m}$, where $A_1,\dots,A_m\in U$
    are disjoint sets of positive measure and $0<r_1<\dots <r_m$ are
    in $\mathbb{R}$; by assumption, each $A_i$ is a subset of $Y$.
    Then $\tp(f/B)$ contains formulas describing the existence of $m$
    positive functions $g_1,...,g_m$ with disjoint supports such that
    $f=r_1g_1+\dots+r_ng_m$ and $g_i$ is the characteristic
    function of a subset of $Y$ with $\|g_i\|^p=\mu(A_i)$, for each
    $i=1,\dots,m$. By saturation of $\cU$, if $\tp(h/B)=\tp(f/B)$, then
    $h$ must also be a simple function with the same distribution as
    $f$ over $V$.

    For a general function $f$, the type $\tp(f/B)$ describes the
    existence of a sequence $(f_i:i\in \omega)$ of simple functions
    converging in norm to $f$. If $\tp(h/B)=\tp(f/B)$, the saturation of
    $\cU$ ensures the existence of a sequence of functions $(h_i:i\in
    \omega)$ such that $\tp(h_i,h/B)=\tp(f_i,f/B)$ for all $i \in
    \omega$.  Then $(h_i:i\in \omega)$ converges to $h$ in norm and
    $\dist(f_i|V)=\dist(h_i|V)$ for all $i \in \omega$. It follows that
    $\dist(f|V)=\dist(h|V)$.
  \item[\impfirst]
    Assume that $\dist(\bar f|V)=\dist(\bar h|V)$. By quantifier elimination,
    to show that $\tp(\bar f/B)=\tp(\bar h/B)$ it suffices to prove that for
    any $\bar g\in B^l$ and any term $t(\bar x,\bar y)$, we have
    $\|t(\bar f,\bar g)\|^p=\|t(\bar h,\bar g)\|^p$.

    Let $\nu$ be the measure on Borel subsets $D$ of $\mathbb{R}^{n+l}$
    defined by $\nu(D)=\mu\{x\in X: (\bar f,\bar g)(x)\in D\}$. Since
    $\dist(\bar f,\bar g)=\dist(\bar h,\bar g)$, we have that $\nu(D)=\mu\{x\in X:
    (\bar h,\bar g)(x)\in D\}$ for any Borel $D\subset \mathbb{R}^{n+l}$.
    Then, by the change of variable formula,

    $\int_X |t(\bar f(x),\bar g(x))|^p
    d\mu(x)=\int_{\mathbb{R}^{n+l}}|t(\bar r,\bar s)|^pd\nu(\bar r,\bar s)=\int_X
    |t(\bar h(x),\bar g(x))|^p d\mu(x)$.
  \end{cycprf}
\end{proof}

\begin{lem}
  \label{lem:TypeInOutBand}
  Let $\bar f \in \cU^n$ be a tuple, $C \leq \cU$ a Banach sublattice.
  We can write each $f_i$ uniquely as $f_i^1 + f_i^2$ where
  $f_i^1 \in  C^{\perp\perp}$ and $f_i^2 \in C^\perp$.
  Then $\tp(\bar f/C)$ determines and is determined by the pair
  $\tp(\bar f^1/C),\tp(\bar f^2)$.

  In other words, for $\bar f,\bar g \in \cU^n$ we have
  $\bar f \equiv_C \bar g$ if and only if
  both $\bar f^1 \equiv_C \bar g^1$ and $\bar f^2 \equiv \bar g^2$.
\end{lem}
\begin{proof}
  Notice that an automorphism $\theta\in\Aut(\cU/C)$
  of $\cU$ which fixes $C$ pointwise must
  fix $C^{\perp\perp}$ and $C^\perp$ set-wise.
  Thus, if such an automorphism sends $\bar f$ to $\bar g$ it must
  also send  $\bar f^1$ to $\bar g^1$ and $\bar f^2$ to $\bar g^2$.

  Conversely, assume $\bar f^1 \equiv_C \bar g^1$
  and $\bar f^2 \equiv \bar g^2$.
  Since $\bar f^2$ and $\bar g^2$ are both in $C^\perp$, it follows by
  quantifier elimination that $\bar f^2 \equiv_C \bar g^2$.
  Thus we have $\theta^1,\theta^2 \in \Aut(\cU/C)$ such that
  $\theta^1\colon \bar f^1 \mapsto \bar g^1$ and $\theta^2\colon \bar f^2 \mapsto \bar g^2$.
  Define $\theta$ by letting it act as $\theta^1$ on $C^{\perp\perp}$ and
  as $\theta^2$ on $C^\perp$, so $\theta \in \Aut(\cU/C)$ and $\theta(\bar f) = \bar g$.
\end{proof}

\begin{lem}
  \label{lem:TupleType}
  Let $\bar f \in \cU^n$ be a tuple, $C \leq \cU$ a Banach sublattice.
  Then $\tp(\bar f/C)$ depends only on the mapping associating to each
  term $t(\bar x)$ in $n$ free variables
  the type $\tp(t(\bar f)/C)$.

  In other words, for $\bar f,\bar g \in \cU^n$ we have
  $\bar f \equiv_C \bar g$ if and only if
  $t(\bar f) \equiv_B t(\bar g)$ for every term $t$.
\end{lem}
\begin{proof}
  Assume $\bar f \not\equiv_C \bar g$.
  We will follow the notation of \fref{lem:TypeInOutBand}.

  If $\bar f^1 \not\equiv_C \bar g^1$ then by \fref{prp:TypeIsDist}
  there is a tuple
  $\bar \lambda \in \bR^n$ such that $\sum \lambda_if_i^1 \not\equiv_C \sum \lambda_ig_i^1$.
  Notice that $\sum \lambda_if_i^1 = \left(\sum \lambda_if_i\right)^1$, etc., so by
  \fref{lem:TypeInOutBand} we have
  $\sum \lambda_if_i \not\equiv_C \sum \lambda_ig_i$.

  If $\bar f^1 \equiv_C \bar g^1$ then by \fref{lem:TypeInOutBand} we
  have $\bar f^2 \not\equiv \bar g^2$.
  By quantifier elimination there is a term $t$ such that
  $\|t(\bar f^2)\| \neq \|t(\bar g^2)\|$.
  Again we have $t(\bar f^2) = t(\bar f)^2$, etc.,
  so: $\|t(\bar f)\|^p = \|t(\bar f^1)\|^p + \|t(\bar f^2)\|^p$ and
  similarly for $\bar g$.
  But $\bar f^1 \equiv_C \bar g^1$ implies that
  $\|t(\bar f^1)\| = \|t(\bar g^1)\|$,
  so $\|t(\bar f)\| \neq \|t(\bar g)\|$.
  Thus $t(\bar f) \not\equiv_C t(\bar g)$.
\end{proof}

\begin{dfn}
  Let $A \subset \cU$ and $f \in \cU$. We say $f$ is in the\
  \emph{definable closure} of $A$ and write $f\in \dcl(A)$ if for
  any automorphism $\Phi\in \Aut(\cU)$, if $\Phi$ fixes $A$ pointwise
  then $\Phi(f)=f$.
\end{dfn}

\begin{fct}
Let $A \subset \cU$. The definable closure of $A$ in $\cU$ is the
sublattice of $\cU$ generated by $A$.
\end{fct}

\begin{lem}
Let $f \in \cU$ and let $A\subset \cU$. If $f\not \in
\dcl(A)$, then the set of realizations of $\tp(f/A)$ is large;
that is, it has cardinality greater than or equal to $\kappa$.
\end{lem}

\begin{proof}
  Let $f\in \cU$, $A\subset \cU$.
  As $\dcl(A)$ is a sublattice of $\cU$, we may assume that
  $\dcl(A)=L_p(Y,V,\mu)$, $(Y,V,\mu)\subset (X,U,\mu)$ and
  $\cU=L_p(X,U,\mu)$.
  If $f\not \in \dcl(A)^n$ then by Lemma \ref{BR-P} we can find
arbitrarily many elements $f'\subset \cU$ such that $f'\neq f$ and
$\dist(f|V)=\dist(f'|V)$. By the previous facts, this shows that the
set of realizations of  $\tp(f/A)$ is large.
\end{proof}

Finally we show stability. We first recall two definitions:

\begin{dfn}
Let $A \subset \cU$; consider $\bar f,\bar g \in \cU^n$ and set $t =
\tp(\bar f/A)$ and $s = \tp(\bar g/A)$.  We define $d(t,s)$ to be the
infimum of all distances $\max\{\| f_i' - g_i'\| \colon 1 \leq i
\leq n \}$ where $\bar f',\bar g' \in \cU^n$, $t = \tp(\bar f'/A)$ and $s =
\tp(\bar g'/A)$.
This defines a metric on the space of $n$-types over
$A$.
\end{dfn}

\begin{dfn} \cite{Iovino:StableBanach}
  Let $\lambda<\kappa$ be an infinite cardinal. We say that
  $\Th_\cA(\cU)$ is \emph{$\lambda$-stable}
  (or \emph{metrically $\lambda$-stable})
  if for any $A\subset \cU$ of cardinality $\leq \lambda$, there is a
  subset of the spaces of types over $A$ that is dense with respect to
  the metric on the space of types.
\end{dfn}

\begin{thm}[Henson \cite{Henson:SeparableBanachSpaces}]
The theory $\Th_A(\cU)$ is $\omega$-stable.
\end{thm}

\begin{proof}
Let $A\subset \cU$ be countable infinite. Then $\dcl(A)$ is a
sublattice of $\cU$ and thus we can find measure spaces
$(Y,V,\mu)\subset (X,U,\mu)$ such that $\dcl(A)=L_p(Y,V,\mu)$ and
$\cU=L_p(X,U,\mu)$.  We have to prove that the density character of
the space of types of functions in the band orthogonal to $A$ is
$\omega$ and that the density character of the space of types in
the band generated by $A$ is $\omega$.

Let $g \in \cU$ be such that $g\in \dcl(A)^{\perp}$. The type
$\tp(g/A)$ is determined by $\|g^+\|$ and $\|g^-\|$. Let $B,C\in U$
be disjoint from $Y$ and from each other, each of measure one. The
set $\{\tp(c_1\chi_{B}-c_2\chi_{C}): c_1,c_2\in \mathbb{Q}^+\}$ is a
countable dense subset of the space of types of functions in the
band orthogonal to $\dcl(A)$.

We can identify $L_p(Y,V,\mu)$ with its canonical image in the space
 $L_p((Y,V,\mu)\otimes([0,1],\cB,m))$, where $([0,1],\cB,m)$ is the
standard Lebesgue space. Let $f \in \cU$ be an element in the band
generated by $\dcl(A)$. By Fact \ref{BR}, we can find $f'\in
L_p((Y,V,\mu) \otimes([0,1],\cB,m))$ such that $\tp(f/A)=\tp(f'/A)$.
To find the density character of the space of types in the band
generated by $A$ it suffices to find the density character of the
space of types over $\dcl(A)$ of elements in
$L_p((Y,V,\mu)\otimes([0,1],\cB,m))$. Since
$L_p((Y,V,\mu)\otimes([0,1],\cB,m))$ is separable, the density
character of the space of types over $\dcl(A)$ is also $\omega$.
\end{proof}

\begin{rmk}[$L_p$ spaces with atoms]
Let $(X,U,\mu)$ be a measure space with atoms such that
$L_p(X,U,\mu)$ is infinite dimensional.  We discuss briefly how
the preceding results in this section can be used to analyze types
and prove $\omega$-stability for $\Th_A(L_p(X,U,\mu))$.

Let $\cU$ be a $\kappa$-universal domain for $\Th_A(L_p(X,U,\mu))$.
By \cite[Example 13.4]{Henson-Iovino:Ultraproducts}, there exists a measure space $(Y,V,\nu)$
such that $\cU$ is isomorphic to $L_p(Y,V,\nu)$ as Banach lattices.
Using \cite[Theorem 2.2]{Henson:SeparableBanachSpaces} one can show that the number of atoms
in $V$ is the same as the number of atoms in $U$, if that number is
finite, and otherwise both $\sigma$-algebras have an infinite number
of atoms.  As discussed in Section 2, we may write $Y$ as the
disjoint union of two measurable sets, $Y = Y_0 \cup Y_1$, with
$Y_0$ being the union (up to null sets) of all the atoms of $V$ and
$Y_1$ being atomless. Since $\cU$ is at least $\omega_1$-saturated,
it is easy to show $\nu(Y_1)>0$.

For each $i = 0,1$, let $V_i,\nu_i$ denote the restrictions of
$V,\nu$ to $Y_i$, and let $\cU_i = L_p(Y_i,V_i,\nu_i)$.  Then we have
the $\ell_p$ direct sum decomposition $\cU \cong \cU_0 \oplus_p \cU_1$
as Banach lattices.  Furthermore, every Banach lattice automorphism
of $\cU$ leaves the sublattices $\cU_0$ and $\cU_1$ invariant; hence
the automorphisms of $\cU$ are exactly the maps $\sigma_0 \oplus
\sigma_1$ obtained as the direct sum of automorphisms $\sigma_0$ of
$\cU_0$ and $\sigma_1$ of $\cU_1$.  The atomic $L_p$ space $\cU_0$ is
isomorphic to the sequence space $\ell_p(S)$ for a suitable set S.
Its Banach lattice automorphisms arise from permutations of $S$.
Using \cite{Henson:SeparableBanachSpaces,Henson-Iovino:Ultraproducts}
as above, we may assume that $\cU_1$ is a
$\kappa$-universal domain for its theory.  Thus $\cU_1$ has a rich
group of Banach lattice automorphisms corresponding to the
equivalence relations defined by types, as discussed previously in
this section.

It is now easy to use automorphisms of $\cU$ to make estimates of the
sizes of type spaces, and thus verify that $\Th_A(L_p(X,U,\mu))$ is
$\omega$-stable.
\end{rmk}

\vfill

\section{Dividing}

Since the theory of $L_p$ Banach lattices is stable, we know it admits
a notion of independence defined by non-dividing.
Let us recall the definition:
\begin{dfn}
  Let $p(x,B)$ be a partial type over $B$ in a possibly infinite tuple
  of variables $x$ (so $p(x,y)$ is a partial type without parameters).
  Then $p(x,B)$ \emph{divides} over another set $C$ if there exists a
  $C$-indiscernible sequence $(B_i\colon i < \omega)$ in $\tp(B/C)$ such
  that $\bigcup_{i<\omega} p(x,B_i)$ is inconsistent.
  \\
  If $A,B,C$ are any sets in a universal domain $\cU$, such that
  $\tp(A/BC)$ does not divide over $C$, then we say that $A$ is
  \emph{independent} from $B$ over $C$, in symbols $A \ind_C B$.
\end{dfn}

This definition of non-dividing yields a natural notion of
independence in every stable theory, and more generally in every
simple one. The goal of this section is to give a more natural
characterization of non-dividing in the context of $L_p$ Banach
lattices. We will prove that it coincides with $*$-independence
(introduced in the next definition) by showing that this relation
has the standard properties of dividing independence.  (See
Proposition \ref{prp:*-indep-properties} below.)

\begin{dfn}
  \label{dfn:SInd}
  Let $A,B,C \leq \cU$ be sublattices of $\cU$ such that
  $C \leq A \cap B$.
  Let $\bE_B$ and $\bE_C$ be the conditional expectation projections
  to $B$ and $C$, respectively, as in
  \fref{ntn:AbstractCondExp}.
  We say that $A$ is \emph{$*$-independent} from $B$ over $C$, in
  symbols $A \ind[*]_C B$, if $\bE_B(f) = \bE_C(f)$ for all $f \in A$.
  \\
  If $A,B,C$ are any subsets of $\cU$, we say that $A \ind[*]_C B$ if
  $A' \ind[*]_{\bar C} B'$, where $A' = \dcl(AC)$ is the
  sublattice generated by $AC$, $\bar C = \dcl(C)$ and
  $B' = \dcl(BC)$.
\end{dfn}

First we have to point out that if we remove the requirement that
$C$ be contained in $A$ we get a weaker (and wrong) definition (see
\prettyref{exm:SIndACDisj}).
Therefore transitivity of $\ind[*]$ does not follow as obviously from
the definition as may seem at first sight.
However, we may replace the requirement that $C \leq A$ with the
following weaker one:

\begin{dfn}\label{dfn:intersect-well}
  Let $A,C \leq \cU$ be sublattices.
  We say that $A$ and $C$ \emph{intersect well} if
  $A^{\perp\perp} \cap C^{\perp\perp} = (A\cap C)^{\perp\perp}$.
  (Clearly $\supset$ always holds.)
\end{dfn}

\begin{rmk}\label{rmk:intersect-well}
It is easy to show (in the notation of the previous definition) that
$A$ and $C$ intersect well if and only if there exists a measure
space $(X,U,\mu)$ such that $\cU \cong L_p(X,U,\mu)$, with measure
subspaces $(Z,W,\mu)$ and $(Y,V,\mu)$ such that $Z \cap Y$ is in $W
\cap V$ and under this isomorphism $ A \cong L_p(Z,W,\mu)$ and $C
\cong L_p(Y,V,\mu)$.
\end{rmk}

\begin{lem}
  \label{lem:SIndBand}
  Let $A,B,C \leq \cU$ be sublattices such that
  $C \leq B$, $A$ and $C$ intersect well,
  and $\bE_C{\restriction}_A = \bE_B{\restriction}_A$.
  Then $A^{\perp\perp} \cap B^{\perp\perp} = (A\cap C)^{\perp\perp}$ (so in particular $A$ and
  $B$ intersect well).
\end{lem}
\begin{proof}
  Let $D = A\cap C$. The inclusion $\supset$ is immediate, so we
  prove $\subset$. Assume not, and let $f \in (A^{\perp\perp} \cap
  B^{\perp\perp}) \setminus D^{\perp\perp}$ be positive. Since $A$
  and $C$ intersect well and $f \in A^{\perp\perp}$, we necessarily
  have $f \notin C^{\perp\perp}$. Replacing $f$ with its restriction
  to $C^\perp$, we may assume that
  $$0 \neq f \in A^{\perp\perp} \cap B^{\perp\perp} \cap C^{\perp}.$$
  As $f \in A^{\perp\perp}$, there is $g \in A$ positive such that
  $\lim_{n\to\infty} (ng)\wedge f = f$.
  Then we also have $\lim_{n\to\infty} \bE_B((ng)\wedge f) = \bE_B(f) > 0$,
  whereby $\bE_B(g\wedge f) > 0$.
  Replacing $g$ with its restriction to the band $D^\perp$, we still have
  $g \in A$ (since $A \geq D$) and $g\wedge f$ is unchanged (since
  $f \in D^\perp$).
  As $A$ and $C$ intersect well: $g \in A\cap D^\perp \subset C^\perp$.

  We now have:
  \begin{gather*}
    \bE_B(g) \geq \bE_B(g\wedge f) > 0 = \bE_C(g).
  \end{gather*}
  This contradicts the assumption.
\end{proof}

\begin{rmk}
  \label{rmk:SIndPres} When $A,B,C \leq \cU$, $C \leq B$ and
  $A^{\perp\perp} \cap B^{\perp\perp} = (A\cap C)^{\perp\perp}$, we
  can represent $\cU$ as $L_p(X,U,\mu)$ such that the sublattices
  $A$, $B$, $C$ and $A\cap C$ are all the $L_p$ spaces of
  sub-measure spaces of $X$.
\end{rmk}

\begin{lem}
  \label{lem:SIndTrans}
  Let $A,B,C \leq \cU$ be sublattices such that
  $C \leq B$ and $A$ and $C$ intersect well.
  Then $A \ind[*]_C B$ if and only if
  $\bE_B{\restriction}_A = \bE_C{\restriction}_A$.
\end{lem}
\begin{proof}
  Let $A' = \dcl(AC)$ and $D = A\cap C$.
  The left to right is immediate since $A \subset A'$.
  We prove right to left.

  First, by \prettyref{lem:SIndBand} we have $A^{\perp\perp} \cap
  B^{\perp\perp} = D^{\perp\perp}$. It follows that there exists a
  measure space $(X,U,\mu)$ such that $\cU \cong L_p(X,U,\mu)$, and
  it has measure subspaces such that under this isomorphism $A \cong
  L_p(Z,W,\mu)$ and $C \cong L_p(Y,V,\mu)$.  See Remark
  \ref{rmk:intersect-well}. Then $A' \cong L_p(Y\cup Z,\langle V\cup
  W\rangle,\mu)$.

  Let $P \in W$ and $Q \in V$.
  Then from the various assumptions we made we obtain:
  \begin{align*}
    \bE_B(\chi_P) & = \bE_C(\chi_P)\\
    \bE_B(\chi_Q) & = \chi_Q = \bE_C(\chi_Q)\\
    \bE_B(\chi_{P\cap Q}) & = \chi_Q\bE_B(\chi_P) = \chi_Q\bE_C(\chi_P)
= \bE_C(\chi_{P\cap Q})
  \end{align*}
  It follows that $\bE_B(\chi_R) = \bE_C(\chi_R)$ for all $R \in
  \langle V\cup W\rangle$, whereby $\bE_B{\restriction}_{A'} =
  \bE_C{\restriction}_{A'}$, as required.
\end{proof}

\begin{cor}
  \label{cor:SIndTrans}
  Let $A,B,C,D \leq \cU$ such that $B \leq C \leq D$.
  Then $A \ind[*]_B D$ if and only if
  $A \ind[*]_B C$ and $A \ind[*]_C D$.
\end{cor}
\begin{proof}
  Replacing $A$ with $\dcl(AB)$, we may assume that $A \geq B$.

  If $A \ind[*]_B C$ and $A \ind[*]_C D$,
  then clearly $\bE_D{\restriction}_A = \bE_B{\restriction}_A$,
  whereby $A \ind[*]_B D$.

  Conversely, assume that $A \ind[*]_B D$.
  Then $\bE_D{\restriction}_A = \bE_B{\restriction}_A = \bE_C{\restriction}_A$.
  If follows by definition that $A \ind[*]_B C$.
  Also, by \prettyref{lem:SIndBand}, $A$ and $C$ intersect well,
  whereby $A \ind[*]_C D$ using \prettyref{lem:SIndTrans}.
\end{proof}

To prove symmetry of $\ind[*]$, we first point out that the
following is a special case of \prettyref{lem:SIndBand}:

\begin{cor}
  \label{cor:SIndBand}
  Let $A,B,C \leq \cU$ be sublattices, such that
  $C \leq A \cap B$, and $A \ind[*]_C B$.
  Then $A^{\perp\perp} \cap B^{\perp\perp} = C^{\perp\perp}$.
\end{cor}

It is therefore harmless to assume, when proving symmetry,
that $A^{\perp\perp} \cap B^{\perp\perp} = C^{\perp\perp}$.

\begin{lem}
  \label{lem:SIndSym}
  Let $A,B,C \leq \cU$ be sublattices such that
  $C \leq A \cap B$ and $A^{\perp\perp} \cap B^{\perp\perp} = C^{\perp\perp}$.
  Using \prettyref{rmk:SIndPres},
  choose $(Y,V,\mu) \subset (Z_i,W_i,\mu) \subset (X,U,\mu)$ such that
  $\cU \cong L_p(X,U,\mu)$, $A \cong L_p(Z_0,W_0,\mu)$,
  $B \cong L_p(Z_1,W_1,\mu)$, $C \cong L_p(Y,V,\mu)$.
  Then the following are equivalent:
  \begin{enumerate}
  \item $A \ind[*]_C B$.
  \item For every $P_0 \in W_0$, $P_1 \in W_1$:
    \begin{gather*}
      \bE_{C}(P_0\cap P_1) = \bE_{C}(P_0)\bE_{C}(P_1)
    \end{gather*}
  \end{enumerate}
\end{lem}
\begin{proof}
  Assume first that $A \ind[*]_C B$.
  Then for every pair of $P_i \in W_i$ and $Q \in V$:
  \begin{align*}
    \int_Q \chi_{P_0\cap P_1}\,d\mu &
    = \int_Q \bE_C(\bE_B(\chi_{P_0}\chi_{P_1}))\,d\mu
    = \int_Q \bE_C(\chi_{P_1}\bE_B(\chi_{P_0}))\,d\mu\\
    & = \int_Q \bE_C(\chi_{P_1}\bE_C(\chi_{P_0}))\,d\mu
    = \int_Q \bE_C(\chi_{P_0})\bE_C(\chi_{P_1})\,d\mu.
  \end{align*}
  Whereby $\bE_{C}(P_0\cap P_1) = \bE_{C}(P_0)\bE_{C}(P_1)$.

  Conversely, assume that $\bE_{C}(P\cap Q) = \bE_{C}(P)\bE_{C}(Q)$
  for every $P \in W_0$ and $Q \in W_1$.
  Then:
  \begin{align*}
    \int_Q \chi_P\,d\mu
    & = \int \chi_{P\cap Q}\,d\mu = \ldots
  \end{align*}
  As $A^{\perp\perp} \cap B^{\perp\perp} = C^{\perp\perp}$ we have
$P\cap Q \subset Y$, whereby:
  \begin{align*}
    \ldots
    & = \int \bE_C(\chi_{P\cap Q})\,d\mu
    = \int \bE_C(\chi_P)\bE_C(\chi_Q)\,d\mu\\
    & = \int \bE_C(\bE_C(\chi_P)\chi_Q)\,d\mu
    = \int \bE_C(\chi_P)\chi_Q\,d\mu\\
    & = \int_Q \bE_C(\chi_P)\,d\mu.
  \end{align*}
  As this holds for all $Q \in W_1$ we get
  $\bE_B(\chi_P) = \bE_C(\chi_P)$, and by standard arguments it follows that
  $\bE_B(f) = \bE_C(f)$ for all $f \in A$.
\end{proof}

\begin{prp}\label{prp:*-indep-properties}
  The relation $\ind[*]$ satisfies the following properties (here $A$,
  $B$, etc., are any small subsets of\ \ $\cU$):
  \begin{enumerate}
  \item Invariance under automorphisms of\ \ $\cU$.
  \item Symmetry: $A \ind[*]_C B \Longleftrightarrow B \ind[*]_C A$.
  \item Transitivity: $A \ind[*]_C BD$ if and only if
    $A \ind[*]_C B$ and $A \ind[*]_{BC} D$.
  \item Finite Character: $A \ind[*]_C B$ if and only
    $\bar a \ind[*]_C B$ for all finite tuples $\bar a \in A$.
  \item Extension: For all $A$, $B$ and $C$ we can find $A'$ such that
    $A \equiv_C A'$ and $A' \ind[*]_C B$.
  \item Local Character: If $\bar a$ is any finite tuple, then there
    is $B_0 \subset B$ at most countable such that
    $\bar a \ind[*]_{B_0} B$.
  \item Stationarity of types: If $A \equiv_C A'$, $A \ind[*]_C B$, and
    $A' \ind[*]_C B$ then $A \equiv_{BC} A'$.
  \end{enumerate}
\end{prp}
\begin{proof}
  \begin{enumerate}
  \item The definition of $\ind[*]$ makes this clear.
  \item Follows directly from
    \prettyref{cor:SIndBand} and \prettyref{lem:SIndSym}.
  \item This is just a rephrasing of \prettyref{cor:SIndTrans}.
  \item One direction is clear.
    Conversely, assume that $A \nind[*]_C B$, so there is
    $f \in \dcl(AC)$ such that $\bE_B(f) \neq \bE_C(f)$.
    This $f$ is simply the limit of terms in members of
    $A\cup C$, so a finite tuple $\bar a \in A$ (and all of $C$) would
    suffice to get some $f'$ which is close enough to $f$ so that
    $\bE_B(f') \neq \bE_C(f')$.
    Then $\bar a \nind[*]_C B$.
  \item
    We may assume that $A,B,C$ are sublattices of $\cU$ and $C \leq
    A\cap B$. By finite character, symmetry and transitivity, it
    suffices to prove the result when $A$ is finitely generated over
    $C$, say $A=\dcl(\{f_1,\dots f_n\}\cup C)$. Furthermore, we may
    assume that there is $l\leq n$ such that $f_i\in C^{\perp
    \perp}$ if $i\leq l$ and $f_i\in C^{\perp}$ if $i>l$. First let
    $g_{l+1},\dots,g_{n}\in B^{\perp}$ with
    $\tp(g_{l+1},\dots,g_{n})=\tp(f_{l+1},\dots,f_{n})$. (See
    \fref{lem:TypeInOutBand}.) By Corollary \ref{existence},
    \fref{prp:TypeIsDist} and Lemma \ref{lem:SIndSym}
    we can find elements $g_1,\dots,g_l\in C^{\perp \perp}$ such
    that $\tp(g_1,\dots,g_l/C)=\tp(f_1,\dots,f_l/C)$ and
    $\{g_1,\dots,g_l\}\ind[*]_{C}{B}$. Let $A'=\dcl(\{g_1,\dots,
    g_n\}\cup C)$. Then ${A'}\ind[*]_{C}{B}$ and $A \equiv_C A'$.
\item Let $\bar B$ be the sublattice generated by $B$.
    Let $C_0 \leq \bar B$ be the sublattice generated by
    $\{\bE_{\bar B}(f)\colon f \in \dcl(\bar a)\}$.
    Clearly $C_0$ is separable, whereby $A_0 = \dcl(\bar aC_0)$ is
    also separable.
    Also, $A_0$ and $B$ intersect well, so letting
    $C_1$ be the lattice generated by $\{\bE_{\bar B}(f)\colon f \in A_0\}$
    we get $\bar a \ind[*]_{C_1} B$.
    Then $C_1 \leq \bar B$ is also separable, so there is a countable
    subset $B_0 \subset B$ such that $C_1 \subset \dcl(B_0)$.
    By transitivity: $\bar a \ind[*]_{B_0} B$ as required.
  \item Again we may assume that all are lattices and $C \leq A\cap B$.
    Then the conditional distribution of members of $A$ over $C$,
    along with the fact that they have the same conditional
    expectation over $C$ and over $B$, determines their conditional
    distribution over $B$.
    \qedhere
  \end{enumerate}
\end{proof}

It follows by \cite[Theorems~1.51,2.8]{BenYaacov:SimplicityInCats}:

\begin{thm}
  The theory of $L_p$ Banach lattices is stable, and non-dividing
  coincides with $*$-independence (i.e.,
  $A \ind_C B \Longleftrightarrow A \ind[*]_C B$).
\end{thm}

The following is a nice feature of independence in $L_p$ lattices:
\begin{prp}
  \label{prp:IndSingle}
  Let $A,B,C \subset \cU$ be any sets.
  Then $A \ind_C B$ if and only if $f \ind_C g$ for all $f \in \dcl(A)$
  and $g \in \dcl(B)$.
  In fact, it suffices to assume that $f \ind_C g$ for every $f,g$
  which are obtained as terms in members of $A$ and $B$,
  respectively.
\end{prp}
\begin{proof}
  Left to right is clear, so we prove right to left.
  By the finite character of independence we may assume that $A$ is a finite set, and
  enumerate it in a tuple $\bar f$.
  Using symmetry, it would suffice to assume that $h \ind_C B$ for all
  $h \in \dcl(A)$,  and we might as well assume that $B \supseteq C$.
  In particular we have
  $t(\bar f) \ind_C B$ for every term $t$ in the right number
  of variables.

  By \fref{prp:*-indep-properties} (extension) we can find a tuple
  $\bar g$ such that $\bar g \equiv_C \bar f$ and such that in addition
  $\bar g \ind_C B$.
  Thus for every term $t$ we also have
  $t(\bar g) \ind_C B$ and $t(\bar f) \equiv_C t(\bar g)$.
  By stationarity we have $t(\bar f) \equiv_C t(\bar g)$ for every
  term $t$, and by \fref{lem:TupleType}:
  $\bar f \equiv_B \bar g$.
  Thus by invariance we obtain $\bar f \ind_C B$, i.e.,
  $A \ind_C B$ as required.
\end{proof}

The following example show that the requirement that $C \leq A$ in the
definition of $\ind[*]$ cannot be entirely done away with:

\begin{exm}
  \label{exm:SIndACDisj}
  Let us work in $L_p([0,3],\cB,\mu)$, where $\mu$ is the Lebesgue measure.
  Let $C$ consist of all constant functions, $B$ consist of all
  functions which are constant of $[0,2]$ and $(2,3]$, and let $A$
  consist of all scalar multiples of $f = 2\chi_{[0,1]} + \chi_{[2,3]}$.

  The $A,B,C$ are sublattices of the ambient lattice, $C \leq B$, and
  for all members $\alpha f$ of  $A$ (where $\alpha$ is a scalar):
  $$\bE_B(\alpha f) = \bE_C(\alpha f) = \alpha\chi_{[0,3]}.$$
  Nevertheless, we have $A \nind[*]_C B$, since
  $\chi_{[2,3]} \in \dcl(AC)$, and
  $$\bE_B(\chi_{[2,3]}) = \chi_{[2,3]} \neq \frac{1}{3}\chi_{[0,3]} =
  \bE_C(\chi_{[2,3]}).$$
\end{exm}

An interesting feature of Hilbert space and many of its expansions
(see \cite{Berenstein-Buechler:ExpansionsOfHilbertSpaces}) is that non-dividing is ``trivial'' in the following
sense: two sets $A$ and $B$ are independent over $C$ if and only if
for every $a\in A$ and $b\in B$, $a$ is independent from $b$ over
$C$. The Banach lattice $\cU$ is \emph{not} ``trivial'' in that
sense, as it is shown by the following well known example from
probability. (See exercise 9.1 in \cite{Folland:RealAnalysis}.)

\begin{exm}
We work inside the standard Lebesgue space $L_p([0,1],\cB,\mu)$. Let
$C=\chi_{[0,1]}$. Let
$a_1=\chi_{[0,1/4]\cup[1/2,3/4]}$, $a_2=\chi_{[0,1/4]\cup[3/4,1]}$,
$a_3=\chi_{[0,1/2]}$. Then $a_j \ind_{C} a_3$ for $j=1,2$ but
$a_1,a_2 \nind_{C} a_3$.
\end{exm}

\begin{rmk}\label{rmk:p-independent}
  As a closing remark in this section, we note that for bounded
  functions over sets of finite measure, dividing independence in
  $L_p$-spaces does not depend at all on $p$.  Specifically, if
  $(X,U,\mu)$ is a measure space with $\mu(X) < \infty$ and $A,B,C
  \subset L_\infty(X,U,\mu)$, then for any $1 \leq p,q < \infty$,
  $A \ind_{C} B$ holds in $L_p(X,U,\mu)$ if and only if it holds in
  $L_q(X,U,\mu)$.
\end{rmk}

\section{Conditional slices}
\label{sec:CondSlice}

In this section we would like to study types and independence a little
further.
First, we would like to give a concrete characterization of types over
a set $C$.
For this purpose we may always assume that $C = \dcl(C)$, i.e., that
$C$ is a Banach sublattice of the ambient model.
We have in fact already given such a characterization of types as
conditional distributions in \fref{prp:TypeIsDist}.
However this characterization depends on a particular presentation of
$C$ as an $L_p$ space and is not intrinsic to the type.

We find our characterization of independence using conditional
expectations similarly deficient as it depends on a good
intersection.
Indeed \fref{exm:SIndACDisj} shows that for lattices $C \leq B$,
comparing conditional expectations over $C$ and over $B$ does not
necessarily suffice to decide whether $A \ind_C B$.
We should therefore like to have
a finer tool that can give an exact measure
of the dependencies of $A$ with $B$ (and with $C$).

We solve both issues using the notion of \emph{conditional slices}.
More precisely, the conditional slices of a single function $f$
over a lattice $C$ yield an
intrinsic characterization of the type $\tp(f/C)$.
We will show that for $C \leq B$, the conditional slices of $f$ over $B$
and $C$ agree if
and only if $f \ind_C B$.
By \fref{prp:IndSingle} this suffices to characterize when
$A \ind_C B$ where $A$ is an arbitrary lattice (i.e., not necessarily
intersecting $C$ well).

If $A$ is a Banach lattice then $A^+$ denotes its positive cone
$A^+ = \{f \in A\colon f \geq 0\}$.

Throughout, $C \leq \cU$ will denote a Banach
sublattice of the ambient model.
We may sometimes wish to fix a presentation
of $C \leq \cU$ as the $L_p$ spaces of
$(Y,V,\mu) \subset (X,U,\mu)$.
We start with a simple observation:
\begin{lem}
  \label{lem:RSlicePresIndep}
  Let $f \in \cU$, $r \in [0,1]$.
  Fixing a presentation of $C$ as above, let
  $R = \{x\colon f(x) \leq 0\}$.
  Then the property $\bP_C(f\leq0) = \bP_C(R) \geq r\chi_Y$ does not depend on
  the chosen presentation.
  We will therefore simply write it as
  ``$\bP_C(f \leq 0) \geq r$''.
\end{lem}
\begin{proof}
  The equivalent property
  $\bP_C(f>0) \leq (1-r)\chi_Y$ holds if and only if for all
  $g \in C^+$ and all $n < \omega$:
  $\|(nf)^+ \wedge g\| \leq \sqrt[p]{1-r}\|g\|$.
\end{proof}

We may therefore conveniently work with any fixed presentation of $C$
as an $L_p$ space, while at the same time keeping our constructions
independent of this presentation.
For
$f \in \cU^+$ and $r \in (0,1)$ we may define, independently of the
presentation of $C$:
\begin{gather*}
  S_r(f) = \{g \in C^+\colon \bP_C(g \leq f) \geq r' \text{ for some } r'> r\}.
\end{gather*}
Assume $g \in S_r(f)$, and let $A = \{x \in X\colon g(x) \leq f(x)\}$.
Then $\chi_Ag \leq f$ whereby
$\int \chi_Ag^p\, d\mu \leq \|f\|^p$.
Since $g \in C$ we also have
$\int \chi_Ag^p\, d\mu = \int \bP[A|C]g^p\, d\mu \geq r'\|g\|^p \geq r\|g\|^p$.
Thus $g \in S_r(f)$ implies $\sqrt[p]{r}\|g\| \leq \|f\|$, whereby
$\|g\| \leq \frac{1}{\sqrt[p]{r}}\|f\|$ for all $g \in S_r(f)$, and thus for
all $g \in \overline{S_r(f)}$.
If $g_1,g_2 \in S_r(f)$ then considering separately the sets on which
$g_1 \geq g_2$ and on which $g_1 < g_2$ we see that $g_1 \vee g_2 \in S_r(f)$.
Since the lattice operations are continuous it follows that
$g_1,g_2 \in \overline{S_r(f)} \Longrightarrow g_1\vee g_2 \in \overline{S_r(f)}$.
Now let $(g_n\colon n < \omega) \subseteq \overline{S_r(f)}$ be an increasing sequence
and let $g$ be its pointwise limit.
By Monotone Convergence we have
$\|g\| \leq \frac{1}{\sqrt[p]{r}}\|f\|$, so $g \in C^+$, and by Dominated
Convergence $g_n \to g$ in $L_p$ and $g \in \overline{S_r(f)}$.
In any $L_p$ space, a strictly increasing sequence of positive
functions is strictly increasing in norm and therefore at most
countable.
Putting everything together we conclude that
$\overline{S_r(f)}$ must admit a greatest element.

\begin{dfn}
  \label{dfn:CondSlice}
  Let $C \leq \cU$ be a Banach sublattice, $f \in \cU$,
  $r \in (0,1)$.
  If $f \geq 0$ we define its \emph{conditional $r$-slice over $C$},
  denoted $\bS_r(f/C)$, as the maximal element of
  $\overline{S_r(f)}$.
  In other words, $\bS_r(f/C) \in C^+$ and is the supremum of all
  $g \in C^+$ verifying $\bP_C\big( f \geq g \big) \geq r'>r$.

  For arbitrary $f$ we define
  $\bS_r(f/C) = \bS_r(f^+/C) - \bS_{1-r}(f^-/C)$.
\end{dfn}

If $g_1 \in S_r(f^+)$ and $g_2 \in S_{1-r}(f^-)$ then
by definition there are $r'' < r < r'$ such that:
\begin{gather*}
  \bP_C\big( g_1 \leq f^+ \big) \geq r', \qquad
  \bP_C\big( g_2 \leq f^- \big) \geq 1-r''.
\end{gather*}
Notice that $r' + (1-r'') > 1$, so if we had in addition
$g_1 \wedge g_2 > 0$ we would
obtain $f^+\wedge f^- > 0$ which is impossible.
We conclude that for all $g_1 \in S_r(f^+)$ and $g_2 \in S_{1-r}(f^-)$:
$g_1 \wedge g_2 = 0$.
It follows by continuity that
$\bS_r(f^+/C) \wedge \bS_{1-r}(f^-/C) = 0$, i.e.:
\begin{gather*}
  \bS_r(f^+/C) = \bS_r(f/C)^+, \qquad \bS_{1-r}(f^-/C) = \bS_r(f/C)^-.
\end{gather*}

Observe also that for $f \geq 0$ we have
$S_r(f) = \bigcup_{r'>r} S_{r'}(f)$, and this is an increasing union,
so
\begin{gather*}
  \bS_r(f) = \bigvee_{r'>r} \bS_{r'}(f).
\end{gather*}
In particular $\bS_r(f/C)$ decreases as $r$ increases (for $f\geq0$ and thus for
arbitrary $f$).

Finally observe that $\bS_r(f/C)$ only depends on $f{\restriction}_{C^{\perp\perp}}$.
Moreover, it is unchanged by automorphisms of $\cU$ which fix $C$, so
it only depends on $\tp(f/C)$.
Thus, if $p = \tp(f/C)$ we may
define $\bS_r(p) = \bS_r(f/C)$.

\begin{lem}
  \label{lem:CondSliceDist}
  Let $f \geq 0$, $r \in (0,1)$, $t \geq 0$, and
  fix a presentation $C = L_p(Y,V,\mu)$.
  Then the following subsets of $Y$ are equal up to a null measure
  set:
  $$\big\{ x \in Y\colon \bP_C(f \geq t)(x) \geq r \big\}
  \qquad =_{a.e.} \qquad
  \bigcap_{r' \in (0,r)\cap\bQ} \big\{ x \in Y\colon \bS_{r'}(f/C)(x) \geq t \big\}.$$
\end{lem}
\begin{proof}
  (All equalities and inequalities here are up to a null measure set.)
  Let $A$ and $B$ denote the sets on the left and right hand side,
  respectively.
  Let $B_{r'}$ denote the set inside the intersection, so
  $B = \bigcap_{r' \in (0,r)\cap\bQ} B_{r'}$.
  If $r' < r$ then $t\chi_A \in S_{r'}(f)$ whereby
  $\bS_{r'}(f/C) \geq t\chi_A$.
  Therefore $A \subseteq B_{r'}$ for all $r' < r$, so $A \subseteq B$.
  On the other hand observe that by construction
  $\bP_C\big( f \geq \bS_{r'}(f/C) \big) \geq r'$.
  Therefore $\bP_C\big( f \geq t\chi_B \big) \geq r'$
  for all (rational) $r' < r$, so
  $\bP_C\big( f \geq t\chi_B \big) \geq r$.
  Since $B \in V$, this is the same as saying that
  $\bP_C\big( f \geq t \big) \geq r$ for (almost) all $x \in B$, whence
  $B \subseteq A$.
\end{proof}

\begin{prp}
  \label{prp:CondSliceType}
  For $f,g \in C^{\perp\perp}$:
  $f \equiv_C g$ if and only if
  $\bS_r(f/C) = \bS_r(g/C)$ for all $r \in (0,1)$.

  More generally, for arbitrary $f,g \in \cU$ we have
  $f \equiv_C g$ if and only if
  $\bS_r(f/C) = \bS_r(g/C)$ for all $r \in (0,1)$
  and $\|(f{\restriction}_{C^\perp})^+\| = \|(g{\restriction}_{C^\perp})^+\|$,
  $\|(f{\restriction}_{C^\perp})^-\| = \|(g{\restriction}_{C^\perp})^-\|$
\end{prp}
\begin{proof}
  For the first assertion,
  left to right has already been observed above.
  For right to left, let us fix a presentation
  $C = L_p(Y,V,\mu)$,
  and consider first the case where $f,g \geq 0$.
  By \fref{lem:CondSliceDist} we have
  $\bP_C(f \geq t) = \bP_C(g \geq t)$ for all $t \geq 0$,
  so the conditional distributions of $f$ and $g$ over $V$ are equal:
  $\dist(f|V) = \dist(g|V)$.
  In the general case we have
  $\bS_r(f^+|C) = \bS_r(f|C)^+ = \bS_r(g|C)^+ = \bS_r(g^+|C)$
  and $\bS_r(f^-|C) = \bS_{1-r}(f|C)^- = \bS_{1-r}(g|C)^- = \bS_r(g^-|C)$
  for all $r \in (0,1)$.
  Again by \fref{lem:CondSliceDist}, $\dist(f^+|V) = \dist(g^+|V)$ and
  $\dist(f^-|V) = \dist(g^-|V)$, whereby
  $\dist(f|V) = \dist(g|V)$.
  We conclude that $f \equiv_C g$ using
  \fref{prp:TypeIsDist}.

  The second assertion follows.
\end{proof}

Thus conditional slices provide a system of invariants for
classifying $1$-types over $C$.
Unlike conditional distributions they do not depend on any extraneous
information such as a presentation of $C$ as a concrete $L_p$ space.
We will now see that various properties of types, of which the most
important are distance and independence, can be read off
directly from the conditional slices.

For this purpose we will first construct, for each system for
conditional slices, a canonical realization of the corresponding type
in $C^{\perp\perp}$.
Let $D = C \otimes L_p([0,1],\cB,\lambda)$, where $([0,1],\cB,\lambda)$ is
the standard Lebesgue space.
Given a presentation $C = L_p(Y,V,\mu)$ we can present
$D = L_p(Y \times [0,1],V \otimes \cB,\mu \times \lambda)$.
For $f \in C$ and $g \in L_p([0,1],\cB,\mu)$ the tensor
$f \otimes g \in D$ is just the function
$h(x,y) = f(x)g(y)$.
Alternatively, we can view $D$ as an abstract
$L_p$ lattice in which $C$ embeds via $f \mapsto f \otimes \chi_{[0,1]}$.
We may embed $D$ in $\cU$ over $C$, and we will choose (arbitrarily)
such an embedding.
Notice that then $D \leq C^{\perp\perp}$.

For $f \in \cU$ we define $\bS(f/C) \in D$ by
$\bS(f/C)(x,y) = \bS_y(f/C)(x)$.
As usual, this does not depend on the presentation of $C$ (although it
does of course depend on the particular presentation we chose for
$L_p([0,1])$).
Indeed we have:
\begin{align*}
  \bS(f/C)^+ & = \bS(f^+/C) \\
  & = \bigvee \{\bS_r(f^+/C) \otimes \chi_{[0,r]}\colon r \in \bQ\cap(0,1)\}, \\
  \bS(f/C)^- & = -\bS(-f^-/C) = (\id \otimes \tau)\big(\bS(f^-/C)\big) \\
  & = \bigvee \{\bS_{1-r}(f^-/C) \otimes \chi_{[r,1]}\colon r \in \bQ\cap(0,1)\}.
\end{align*}
Here $\tau \in \Aut(L_p([0,1]))$ consists of reversing the order on $[0,1]$:
$(\tau h)(y) = h(1-y)$.
As before, $\bS(f/C)$ depends only on $\tp(f/C)$, so we may write
it instead as $\bS(p)$ where $p = \tp(f/C)$.

Let $D_{dec} \subseteq D$ be the set of $h(x,y) \in D$ which are decreasing in
$y$.
\begin{lem}
  \label{lem:MergedCondSlice}
  For all $f \in \cU$ we have $\bS(f/C) \in D_{dec}$.
  If in addition $f \in C^{\perp\perp}$ then $\bS(f/C) \equiv_C f$.
  Finally, if $f \in C^{\perp\perp} \cap D_{dec}$ then $\bS(f/C) = f$.
\end{lem}
\begin{proof}
  The first assertion is clear.

  Before we proceed, let us first observe that if
  $f \in D_{dec}^+$, $g \in C^+$ and $r \in (0,1)$ then
  $g \in S_r(f/C)$ if and only if $g \otimes \chi_{[0,r']} \leq f$ for some $r' > r$.
  Thus $\bS_r(f/C) = \bigvee \{g\in C^+\colon g \otimes \chi_{[0,r']} \leq f, r' > r\}$.

  For the second assertion we assume that $f \in C^{\perp\perp}$.
  Let $h = \bS_r(f)$, and consider first the case where $f \geq 0$.
  Then $\bS_r(h/C)$ is equal (by our observation) to
  $\bigvee \{g\in C^+\colon g \otimes \chi_{[0,r']} \leq h, r' > r\}$.
  This is equal by construction of
  $h = \bS(f/C)$ to $\bigvee_{r' > r} \bS_{r'}(f/C) = \bS_r(f/C)$.
  Thus $f$ and $h$ have the same conditional slices and therefore
  the same type (over $C$).
  In the general case this implies that $h^+ \equiv_C f^+$
  and $(\id \otimes \tau)(h^-) \equiv_C f^-$.
  Since $\id \otimes \tau$ is an automorphism of $D$ fixing $C$,
  by quantifier elimination we obtain $h^- \equiv_C f^-$.
  Thus $h \equiv_C f$.

  For the third assertion, let us first
  consider the case $f \in D_{dec}^+$.
  By our observation
  $\bS_r(f/C) = \bigvee \{g\in C^+\colon g \otimes \chi_{[0,r']} \leq f, r' > r\}$,
  so $\bS_r(f/C) \otimes \chi_{[0,r]} \leq f$ and thus $\bS(f/C) \leq f$.
  On the other hand we already know that $f \equiv_C \bS(f/C)$,
  so $\|f\| = \|\bS(f/C)\|$, and together with $0 \leq \bS(f/C) \leq f$ we
  obtain $f = \bS(f/C)$.
  If $f \in D_{dec}$ is negative then
  $(\id\otimes\tau)(-f) \in D_{dec}$ is positive, so
  $\bS(f/C)
  = -(\id\otimes\tau)\bS(-f/C)
  = -(\id\otimes\tau)\bS((\id\otimes\tau)(-f)/C)
  = -(\id\otimes\tau)^2(-f) = f$.
  The general case ensues.
\end{proof}

It follows that not only do conditional slices serve as a complete
system of invariants for types in $C^{\perp\perp}$, but they also allow easy
extraction of various other invariants of such types:
\begin{prp}
  \label{prp:CondSliceExpectNorm}
  For all $f \in \cU$ we have
  \begin{gather*}
    \bE_C(f) = \int_0^1 \bS_r(f/C)\,dr \\
    \|f{\restriction}_{C^{\perp\perp}}\| =
    \left( \int_0^1 \|\bS_r(f/C)\|^p\,dr \right)^{\frac{1}{p}}
  \end{gather*}
  The first integral is just integration of a function of two
  variables: $\bE_C(f)(x) = \int_0^1 \bS_r(f/C)(x)\,dr$ for (almost) all
  $x$.
\end{prp}
\begin{proof}
  Let $h(x,r) = \bS(f/C)(x,r) = \bS_r(f/C)(x)$.
  Then $h \equiv_C f{\restriction}_{C^{\perp\perp}}$, whereby:
  \begin{gather*}
    \bE_C(f)(x) = \bE_C(h)(x) = \int_0^1 h(x,r)\,dr \\
    \|f{\restriction}_{C^{\perp\perp}}\|^p = \|h\|^p
    = \int_0^1 \int_C |h(x,r)|^p\,d\mu(x) \,dr
    = \int_0^1 \|h(\cdot,r)\|^p\,dr
  \end{gather*}
  For both we use Fubini's theorem (and,
  for the first, the definition of conditional expectation).
\end{proof}

\begin{rmk}
  \label{rmk:SecondCondExpUniqueness}
  Since every two presentations of $C$ as a concrete $L_p$ space
  differ by (essentially) no more than a density change, one can
  verify that the function
  $\int_0^1 \bS_r(f/C)(x)\,dr \in C$ does not depend on the
  presentation of $C$, justifying the notation
  $\bE_C(f) = \int_0^1 \bS_r(f/C)\,dr$.
  Alternatively, one may develop a theory of integration of $C$-valued
  functions (and more generally, of $E$-valued functions, where $E$ is
  any Dedekind complete vector lattice), in which case the identity
  $\bE_C(f) = \int_0^1 \bS_r(f/C)\,dr$ holds directly, the right hand side
  being the $C$-valued integral of the mapping $r \mapsto \bS_r(f/C)$.

  Either way this gives an alternative proof to the fact
  (Proposition \ref{prp:CExpUnq}) that the conditional expectation
  mapping $\bE_C\colon \cU \to C$ does not depend on any particular
  choice of presentation for $\cU$ and $C$.
\end{rmk}

We get a similar result for the distance between types, but a little
more work is required.

Let $S_1^{\perp\perp}(C) \subseteq S_1(C)$ denote the set of all types whose
realizations are in $C^{\perp\perp}$
and let $S_1^{\perp}(C)\subseteq S_1(C)$ denote
the set of types whose realizations are in $C^\perp$.

\begin{thm}
  \label{thm:CondSliceIsom}
  For every type $p \in S_1^{\perp\perp}(C)$, $\bS(p)$ is its unique realization
  in $D_{dec}$.
  Thus
  $\bS\colon S_1^{\perp\perp}(C) \to D_{dec}$ is a bijection, whose inverse is
  the mapping $f \mapsto \tp(f/C)$.
  Moreover, equipping $S_1^{\perp\perp}(C)$ with the usual distance between
  types,  this bijection is an isometry.
\end{thm}
\begin{proof}
  The first assertion follows immediately from
  \fref{lem:MergedCondSlice}, so we concentrate on the isometry
  assertion.

  Let $p,q \in S_1^{\perp\perp}(C)$, $f = \bS(p)$, $g = \bS(q)$.
  Then $f \models p$ and $g \models q$ by \fref{lem:MergedCondSlice}, and by
  definition of the distance between types: $d(p,q) \leq d(f,g)$.
  We should now show that if $f' \models p$ and $g' \models q$ are any two other
  realizations then $d(f',g') \geq d(f,g)$.

  Let us fix a presentation of $C \leq \cU$ as the $L_p$ spaces of
  $(Y,V,\mu) \subset (X,U,\mu)$.
  By a density change argument we may assume there is $S \in V$ such
  that $\mu(S) = 1$, and such that $f,g,f',g'$ are in the band generated
  by $\chi_S$ (in fact, for all intents and purposes we may simply assume
  that $Y = S$).

  Having a presentation we may speak of characteristic and simple
  functions.
  Let us first consider the case where both $f$ and $g$ are
  characteristic.
  Notice that the type of $f$ over $C$ says that $f$ is
  characteristic: $0 \leq f \leq \chi_S$ and $f\wedge(\chi_S-f) = 0$.
  Thus we may write $f = \chi_T$, $f' = \chi_{T'}$.
  As $f \in D$ we may identify $T$ with a $(V \otimes \cB)$-measurable subset
  of $Y \times [0,1]$.
  Moreover, $f = \chi_T \in D_{dec}$, so $T$ must be equal to the ``area
  under the graph'' of $\bE_C(f)$:
  $T = \{(x,y) \in Y \times [0,1]\colon y \leq \bE_C(f)(x)\}$.
  On the other hand, $f \equiv_C f' \Longrightarrow \bE_C(f) = \bE_C(f')$.

  We make similar assumptions and observations for
  $g = \chi_R$, $g' = \chi_{R'}$.
  In particular:
  $R = \{(x,y) \in Y \times [0,1]\colon y \leq \bE_C(g)(x)\}$.
  It follows that
  $\bE_C(T \setminus R) = \bE_C(f) \dotminus \bE_C(g)$, while for
  $T',R'$ we only have:
  $\bE_C(T' \setminus R') \geq \bE_C(f') \dotminus \bE_C(g')$, and putting together:
  $\bE_C(T' \setminus R') \geq \bE_C(T \setminus R)$.
  Same holds of course exchanging $T$ and $R$.
  We obtain:
  $d(f',g')^p = \int_Y [\bE_C(T' \setminus R') + \bE_C(R'\setminus T')]d\mu
  \geq \int_Y [\bE_C(T \setminus R) + \bE_C(R\setminus T)]d\mu = d(f,g)^p$.
  \begin{align*}
    d(f',g')^p &
    = \mu(T' \setminus R') + \mu(R' \setminus T') \\ &
    = \int_Y [\bE_C(T' \setminus R') + \bE_C(R'\setminus T')]d\mu \\ &
    \geq \int_Y [\bE_C(T \setminus R) + \bE_C(R\setminus T)]d\mu \\ &
    = d(f,g)^p
  \end{align*}
  Let us now consider the case where $f$ and $g$ are simple positive
  functions with range in $\{0,\ldots,n\}$.
  We can write them in a unique fashion as
  $f = \sum_{i<n} \chi_{T_i}$, $g = \sum_{i<n} \chi_{R_i}$ where $T_0 \subseteq T_1 \subseteq \ldots \subseteq T_{n-1}$ and
  $R_0 \subseteq R_1 \subseteq \ldots \subseteq R_{n-1}$.
  As above the decompositions are coded in the types over $C$,
  so we get corresponding decompositions
  $f' = \sum_{i<n} \chi_{T_i'}$, $g' = \sum_{i<n} \chi_{R_i'}$.
  Since $f,g \in D_{dec}$ we must have $\chi_{T_i},\chi_{R_i} \in D_{dec}$ whereby
  $\chi_{T_i} = \bS(\chi_{T_i'})$, $\chi_{R_i} = \bS(\chi_{R_i'})$.
  As above it follows that
  $\bE_C(T'_i \setminus R'_j) \geq \bE_C(T_i \setminus R_j)$,
  $\bE_C(R'_i \setminus T'_j) \geq \bE_C(R_i \setminus T_j)$.

  In order to calculate $d(f',g')$, let us define
  $c_0 = 1$ and for $n > 0$: $c_n = (n+1)^p - 2n^p + (n-1)^p$.
  As $x \mapsto x^p$ is convex all $c_n$ are positive.
  One shows by induction first that
  $(n+1)^p - n^p = \sum_{i \leq n} c_i$
  and then that $n^p = \sum_{i<n} (n-i)c_i$.
  The last identity can also be written as
  $n^p = \sum_{i\leq j<n} c_{j-i}$.
  It follows that
  $d(f',g')^p = \sum_{i\leq j<n} c_{j-i}[\mu(T'_j \setminus R'_i) + \mu(R'_j\setminus T'_i)]$, and
  similarly for $f,g$.
  Thus:
  \begin{align*}
    d(f',g')^p
    & = \sum_{i\leq j<n} c_{j-i}\int_Y [\bE_C(T'_j\setminus R_i')+\bE_C(R_j'\setminus T_i')]\,d\mu\\
    & \geq \sum_{i\leq j<n} c_{j-i}\int_Y [\bE_C(T_j\setminus R_i)+\bE_C(R_j\setminus T_i)]\,d\mu = d(f,g)^p.
  \end{align*}

  For simple functions with range, say, in
  $\{\frac{m}{n}\colon m \leq n^2\}$, just apply the previous result and
  shrink by a factor of $n$.
  Arbitrary positive $L_p$ functions are increasing limits
  (both pointwise and in $L_p$ norm) of such functions, whence the
  result for positive functions.
  If $f$ and $g$ are possibly negative but bounded from below, say
  $f,g \geq -M\chi_S$, then same hold of $f',g'$ and we have:
  $d(f',g') = d(f'+M\chi_S,g'+M\chi_S) \geq d(f+M\chi_S,g+M\chi_S) = d(f,g)$.
  Since the bounded functions are dense in $L_p$ we obtain the general
  case.
\end{proof}

This can be extended to obtain an explicit expression for the distance
between arbitrary $1$-types over $C$ (i.e., not necessarily of
functions in $C^{\perp\perp}$).
\begin{ntn}
  \label{ntn:TypeManips}
  For $p = \tp(f/C) \in S_1(C)$, let:
  \begin{enumerate}
  \item $p^+ = \tp(f^+/C)$, $p^- =\tp(f^-/C)$.
  \item $\|p\| = \|f\|$.
  \item $p{\restriction}_{C^{\perp\perp}} = \tp(f{\restriction}_{C^{\perp\perp}}/C)$.
  \item $p{\restriction}_{C^\perp} = \tp(f{\restriction}_{C^\perp}/C)$.
  \end{enumerate}
\end{ntn}

\begin{cor}
  \label{cor:TypeDist}
  For all $p,q \in S_1(C)$:
  \begin{align*}
    d(p,q)^p =
    \begin{aligned}[t]
      &\int_0^1 \|\bS_r(p)-\bS_r(q)\|^p\,dr\\
      &\quad+ \big| \|p^+{\restriction}_{C^\perp}\| - \|q^+{\restriction}_{C^\perp}\| \big|^p
      + \big| \|p^-{\restriction}_{C^\perp}\| - \|q^-{\restriction}_{C^\perp}\| \big|^p
    \end{aligned}
  \end{align*}
\end{cor}
\begin{proof}
  Notice that for all $f,g$:
  $\|f-g\|^p =
  \|f{\restriction}_{C^{\perp\perp}} - g{\restriction}_{C^{\perp\perp}}\|^p +
  \|f{\restriction}_{C^{\perp}} - g{\restriction}_{C^{\perp}}|\|^p$,
  so $d(p,q)^p =
  d(p{\restriction}_{C^{\perp\perp}},q{\restriction}_{C^{\perp\perp}})^p
  +  d(p{\restriction}_{C^\perp},q{\restriction}_{C^\perp})^p$.
  By  \fref{thm:CondSliceIsom}:
  $d(p{\restriction}_{C^{\perp\perp}},q{\restriction}_{C^{\perp\perp}})^p
  = \|\bS(p)-\bS(q)\|^p
  = \int_0^1 \|\bS_r(p)-\bS_r(q)\|^p\,dr$.

  We are left with showing that if
  $p,q \in S_1^\perp(C)$ then
  $d(p,q)^p
  = \big| \|p^+\| - \|q^+\| \big|^p
  + \big| \|p^-\| - \|q^-\| \big|^p$.
  If $f \models p$ then $p$ is
  determined by the fact that $f \in C^{\perp}$ and by
  the numbers $\|f^+\|,\|f^-\|$.
  If $g \models q$ then:
  \begin{align*}
    \|f-g\|^p
    & \geq \|f^+ - g^+\|^p + \|f^- - g^-\|^p\\
    & \geq \big| \|f^+\| - \|g^+\| \big|^p
    + \big| \|f^-\| - \|g^-\| \big|^p \\
    & = \big| \|p^+\| - \|q^+\| \big|^p
    + \big| \|p^-\| - \|q^-\| \big|^p.
  \end{align*}
  This lower bound can be attained by
  taking $f^+$ and $g^+$ to be the constants $\|f^+\|$ and $\|g^+\|$,
  respectively, over a set $A$ of measure $1$ (where $\chi_A\in
  C^\perp$) and similarly for $f^-$ and $g^-$ over a disjoint set $B$
  of measure $1$ (i.e. $\chi_B\in {\dcl(\chi_A,C)}^\perp$).
\end{proof}

Finally, we observe that conditional slices yield another
characterization of independence. Indeed, let $C \leq B$, and let $E
= B \otimes L_p([0,1],\cB,\lambda)$.
Then $D = C\otimes L_p([0,1],\cB,\lambda)\leq E$, and clearly $D \ind_C B$.

\begin{lem}
  \label{lem:IndBySliceSingle}
  For all $f \in \cU$, the following are equivalent:
  \begin{enumerate}
  \item $f \ind_C B$.
  \item For all $0<r<1$: $\bS_r(f/B) = \bS_r(f/C)$.
  \item For all $0<r<1$: $\bS_r(f/B) \in C$.
  \end{enumerate}
\end{lem}
\begin{proof}
  First, we may assume that $f \in B^{\perp\perp}$, as replacing $f$ with its
  component in this band leaves all statements unchanged.

  Assume first that $f \ind_C B$, and let
  $f' = \bS(f/C) \in D_{dec} \subseteq E_{dec}$.
  Then $f \in B^{\perp\perp} \Longrightarrow f \in C^{\perp\perp}$, so $f \equiv_C f'$.
  Now $D \ind_C B \Longrightarrow f' \ind_C B$, and by stationarity we get that
  $f' \equiv_B f$.
  As $f' \in E_{dec}$ we must have $f' = \bS(f/B)$, so
  $\bS_r(f/B) = \bS_r(f/C)$ for all $0 < r < 1$.

  Conversely, assume that $\bS_r(f/B) \in C$ for all $0 < r < 1$,
  and let $f' = \bS(f/B) \in E_{dec}$.
  Then $f \equiv_B f'$  and $f' \in D_{dec}$, so $f' \ind_C B$ and
  therefore $f \ind_C B$.
\end{proof}

Using \prettyref{prp:IndSingle}, we conclude:
\begin{prp}\label{prp:independence-via-slices}
  Let $C \leq B \leq \cU$ be sublattices and $A$ any set.
  Then the following are equivalent:
  \begin{enumerate}
  \item $A \ind_C B$.
  \item $\bS_r(f/B) = \bS_r(f/C)$
    for every $r \in (0,1)$ and $f$ which is a
    term in members of $A$.
  \item $\bS_r(f/B) \in C$
    for every $r \in (0,1)$ and $f$ which is a
    term in members of $A$.
  \end{enumerate}
\end{prp}

\section{Canonical bases}

The notion of the \emph{canonical base} of a type comes from general
stability theory.
It is, in a sense, a minimal set of parameters which is required to
define the type.
Since we did not discuss definability of types in this paper we shall
use an alternative approach, namely, viewing the canonical base as a
canonical parameter for the \emph{parallelism class} of the type.
We will try and give a quick introduction to the uninitiated.

We again work inside a $\kappa$-universal domain
$\cU$ for the theory of atomless $L_p$ Banach lattices, and we take
$(X,U,\mu)$ to be a measure space such that $\cU=L_p(X,U,\mu)$.

Since a type over a subset $A \subseteq \cU$ is the same as a type over
$\dcl(A)$, i.e., the Banach sublattice generated by $A$, we will only
consider types over Banach sublattices of $\cU$.
For $A \leq B \leq \cU$, $q \in S_n(B)$ and $p \in S_n(A)$, we say
that $q$ is a \emph{non-forking extension} of $p$ if
$\bar f \models q$ implies $\bar f \models p$ and
$\bar f \ind_A B$.
By \fref{prp:*-indep-properties} a type $p \in S_n(A)$ admits a
\emph{unique} non-forking extension to a type over $B$ (i.e., all
types over sublattices of $\cU$ are \emph{stationary}).
We will use $p\rest^B$ to denote the unique non-forking extension.

The group of automorphisms $\Aut(\cU)$ acts on types over subsets of
$\cU$ naturally, by acting on their parameters.
We wish to distinguish those automorphisms $f\in\Aut(\cU)$ which
\emph{essentially} fix $p$.
In order to compare the two types $p$ and $f(p)$, which may have
distinct domains $A$ and $f(A)$, we compare their unique non-forking
extensions to $A\cup f(A)$.
We say that $p$ and $f(p)$ are \emph{parallel}
if $p\rest^{A\cup f(A)} = f(p)\rest^{A\cup f(A)}$, or equivalently,  if
$p\rest^\cU = f(p)\rest^\cU$, noticing that the latter is always equal
to $f(p\rest^\cU)$.

This leads us to:
\begin{dfn}
  A \emph{canonical base} for a type $p \in S_n(A)$
  is a subset $C \subseteq \cU$ such that an
  automorphism $f \in \Aut(\cU)$ fixes $p\rest^\cU$
  if and only if it fixes each member of $C$.
\end{dfn}
(In a general stable theory we will usually only define canonical
bases for stationary types.)

Notice that $f \in \Aut(\cU)$ fixes $p\rest^\cU$ if and only if it fixes
set-wise the class $\{q \in S_n(B)\colon B \leq \cU, q\rest^\cU = p\rest^\cU\}$, called the
\emph{parallelism class} of $p$.

It follows from the definition that if $C$ and $C'$ are two canonical
bases for $p$ then $\dcl(C) = \dcl(C')$, so it is legitimate in a
sense to speak of \emph{the} canonical base of a type.
In a general stable theory canonical bases of types need not always
exist as sets of ordinary elements as we defined above.
They do exist in general as sets of \emph{imaginary elements}, a
topic which we will not discuss in the present paper (see
\cite[Section~5]{BenYaacov-Usvyatsov:CFO}).

Our goal in this section is to show that in atomless $L_p$ Banach
lattices  canonical bases always exist as sets of ordinary elements
(some would call this having \emph{built-in canonical bases}).
In fact, this has already been essentially proved above in
\fref{sec:CondSlice}.

\begin{thm}
  Let $\bar f \in \cU^n$ be a tuple and  $A \leq \cU$ a sublattice.
  Let
  \begin{gather*}
    \Cb(\bar f/A) =
    \{\bS_r(t(\bar f)/A)\colon r\in (0,1) \text{ and term } t \text{ in } n \text{
      variables}\}.
  \end{gather*}
  Then $\Cb(\bar f/A)$
  only depends on $p = \tp(\bar f/A)$ and is a canonical base
  for $p$.

  In the case where $n = 1$ the set
  $\{\bS_r(f/A)\colon r \in (0,1)\}$ suffices.
\end{thm}
\begin{proof}
  We have $\Cb(\bar f/A) \subseteq A$ by construction.
  Let $C = \dcl\big( \Cb(\bar f/A) \big) \leq A$.
  Then $p$ does not fork over $C$ and
  $\Cb(\bar f/A) = \Cb(\bar f/C)$ by
  \fref{prp:independence-via-slices}, so we might as well assume that
  $C = A$, i.e., that $\Cb(\bar f/A)$ generates $A$.
  Thus, if $\theta \in \Aut(\cU)$ fixes $\Cb(\bar f/A)$ pointwise then it fixes
  $A$ pointwise, so
  $\theta(p) = p$ and therefore $\theta(p\rest^\cU) = p\rest^\cU$.

  Conversely, assume that $\theta(p\rest^\cU) = p\rest^\cU$.
  For a term $t(\bar x)$ let $p^t = \tp(t(\bar f)/A)$,
  noticing that this indeed only depends on $p$, and we may apply the
  same definition to arbitrary $n$-types.
  Observe then that $(p\rest^\cU)^t = (p^t)\rest^\cU$.
  A member of $\Cb(\bar f/A)$ is of the form
  $\bS_r(t(\bar f/A))  = \bS_r(p^t) = \bS_r((p\rest^\cU)^t)$,
  so each is fixed by $\theta$.

  The case $n=1$ is proved similarly using
  \fref{lem:IndBySliceSingle}.
\end{proof}

Notice that it follows that $\Cb(\bar f/A) \subseteq \dcl(A)$ and that
$\bar f \ind_{\Cb(\bar f/A)} A$ by \fref{prp:independence-via-slices}
and \fref{prp:IndSingle}.
Moreover, $\Cb(\bar f/A)$ is minimal as such, in the sense that if
$B \subseteq \dcl(A)$ and $\bar f \ind_B A$ then
$\Cb(\bar f/A) \subseteq \dcl(B)$.
These are indeed properties of canonical bases in an arbitrary stable
theory.

\bibliographystyle{begnac}
\bibliography{begnac}

\providecommand{\bysame}{\leavevmode\hbox to3em{\hrulefill}\thinspace}
\providecommand{\MR}{\relax\ifhmode\unskip\space\fi MR }
\providecommand{\MRhref}[2]{%
  \href{http://www.ams.org/mathscinet-getitem?mr=#1}{#2}
}
\providecommand{\href}[2]{#2}
\begin{thebibliography}{{Ben}03b}

\bibitem[AAB93]{Abramovich-Aliprantis-Burkinshaw:ContractiveProjections}
Y.~A. Abramovich, C.~D. Aliprantis, and O.~Burkinshaw, \emph{An elementary
  proof of {D}ouglas' theorem on contractive projections on {$L\sb 1$}-spaces},
  J. Math. Anal. Appl. \textbf{177} (1993), no.~2, 641--644.

\bibitem[And66]{Ando:ContractiveProjections}
T.~Andô, \emph{Contractive projections in {$L\sb{p}$} spaces}, Pacific J.
  Math. \textbf{17} (1966), 391--405.

\bibitem[BB04]{Berenstein-Buechler:ExpansionsOfHilbertSpaces}
Alexander Berenstein and Steven Buechler, \emph{Simple stable homogeneous
  expansions of {H}ilbert spaces}, Annals of Pure and Applied Logic
  \textbf{128} (2004), 75--101.

\bibitem[BBHU08]{BenYaacov-Berenstein-Henson-Usvyatsov:NewtonMS}
Itaï {Ben Yaacov}, Alexander Berenstein, C.~Ward Henson, and Alexander
  Usvyatsov, \href{http://math.univ-lyon1.fr/~begnac/articles/mtfms.pdf}
  {\emph{Model theory for metric structures}}, Model theory with Applications
  to Algebra and Analysis, volume 2 (Zoé Chatzidakis, Dugald Macpherson, Anand
  Pillay, and Alex Wilkie, eds.), London Math Society Lecture Note Series, vol.
  350, Cambridge University Press, 2008, pp.~315--427.

\bibitem[BDCK67]{Bretagnolle-DacunhaCastelle-Krivine:LoisStables}
J.~Bretagnolle, D.~Dacunha-Castelle, and J.-L. Krivine, \emph{Lois stables et
  espaces {$L\sp{p}$}}, Symposium on {P}robability {M}ethods in {A}nalysis
  ({L}outraki, 1966), Springer, Berlin, 1967, pp.~48--54.

\bibitem[{Ben}03a]{BenYaacov:PositiveModelTheoryAndCats}
Itaï {Ben Yaacov}, \href{http://dx.doi.org/10.1142/S0219061303000212}
  {\emph{Positive model theory and compact abstract theories}}, Journal of
  Mathematical Logic \textbf{3} (2003), no.~1, 85--118.

\bibitem[{Ben}03b]{BenYaacov:SimplicityInCats}
\bysame, \href{http://dx.doi.org/10.1142/S0219061303000297} {\emph{Simplicity
  in compact abstract theories}}, Journal of Mathematical Logic \textbf{3}
  (2003), no.~2, 163--191.

\bibitem[BL03]{Buechler-Lessmann:SimpleHomogeneous}
Steven Buechler and Olivier Lessmann, \emph{Simple homogeneous models}, Journal
  of the American Mathematical Society \textbf{16} (2003), 91--121.

\bibitem[BR85]{Berkes-Rosenthal:AlmostExchangeableSequences}
István Berkes and Haskell~P. Rosenthal, \emph{Almost exchangeable sequences of
  random variables}, Z. Wahrsch. Verw. Gebiete \textbf{70} (1985), no.~4,
  473--507.

\bibitem[BU]{BenYaacov-Usvyatsov:CFO}
Itaï {Ben Yaacov} and Alexander Usvyatsov,
  \href{http://math.univ-lyon1.fr/~begnac/articles/cfo.pdf} {\emph{Continuous
  first order logic and local stability}}, Transactions of the American
  Mathematical Society, to appear,
  \href{http://arxiv.org/abs/0801.4303}{arXiv:0801.4303}.

\bibitem[Dou65]{Douglas:ContractiveProjections}
R.~G. Douglas, \emph{Contractive projections on an {$L\sb{1}$} space}, Pacific
  J. Math. \textbf{15} (1965), 443--462.

\bibitem[Fol84]{Folland:RealAnalysis}
Gerald~B. Folland, \emph{Real analysis}, Pure and Applied Mathematics (New
  York), John Wiley \& Sons Inc., New York, 1984, Modern techniques and their
  applications, A Wiley-Interscience Publication.

\bibitem[Fre04]{Fremlin:MeasureTheoryVol3}
D.~H. Fremlin,
  \href{http://www.essex.ac.uk/maths/staff/fremlin/mt3.2004/index.htm}
  {\emph{Measure theory volume 3: Measure algebras}}, Torres Fremlin, 25 Ireton
  Road, Colchester CO3 3AT, England, 2004.

\bibitem[Hen76]{Henson:NonstandardHulls}
C.~Ward Henson, \emph{Nonstandard hulls of {B}anach spaces}, Israel Journal of
  Mathematics \textbf{25} (1976), 108--144.

\bibitem[Hen87]{Henson:SeparableBanachSpaces}
\bysame, \emph{Model theory of separable {B}anach spaces}, Journal of Symbolic
  Logic \textbf{52} (1987), 1059--1060, abstract.

\bibitem[HI02]{Henson-Iovino:Ultraproducts}
C.~Ward Henson and José Iovino, \emph{Ultraproducts in analysis}, Analysis and
  Logic (Catherine Finet and Christian Michaux, eds.), London Mathematical
  Society Lecture Notes Series, no. 262, Cambridge University Press, 2002.

\bibitem[HLR91]{Haydon-Levy-Raynaud:RandomlyNormedSpaces}
R.~Haydon, M.~Levy, and Y.~Raynaud, \emph{Randomly normed spaces}, Travaux en
  Cours [Works in Progress], vol.~41, Hermann, Paris, 1991.

\bibitem[Iov99]{Iovino:StableBanach}
José Iovino, \emph{Stable {B}anach spaces and {B}anach space structures, {I}
  and {II}}, Models, algebras, and proofs (Bogotá, 1995), Lecture Notes in
  Pure and Appl. Math., vol. 203, Dekker, New York, 1999, pp.~77--117.

\bibitem[Kal02]{Kallenberg:ModernProbability}
Olav Kallenberg, \emph{Foundations of modern probability}, second ed.,
  Probability and its Applications (New York), Springer-Verlag, New York, 2002.

\bibitem[KM81]{Krivine-Maurey:EspacesDeBanachStables}
Jean-Louis Krivine and Bernard Maurey, \emph{Espaces de {B}anach stables},
  Israel Journal of Mathematics \textbf{39} (1981), no.~4, 273--295.

\bibitem[Lac74]{Lacey:IsometricTheory}
H.~Elton Lacey, \emph{The isometric theory of classical {B}anach spaces},
  Springer-Verlag, New York, 1974, Die Grundlehren der mathematischen
  Wissenschaften, Band 208.

\bibitem[LT79]{Lindenstrauss-Tzafriri:ClassicalBanachSpacesII}
Joram Lindenstrauss and Lior Tzafriri, \emph{Classical {B}anach spaces. {II}},
  Ergebnisse der Mathematik und ihrer Grenzgebiete [Results in Mathematics and
  Related Areas], vol.~97, Springer-Verlag, Berlin, 1979, Function spaces.

\bibitem[Pom00]{Pomper:PhD}
Markus Pomper, \emph{Types over {B}anach spaces}, Ph.D. thesis, Université of
  Illinois at Urbana-Champaign, 2000.

\bibitem[Sch74]{Schaefer:BanachLattices}
Helmut~H. Schaefer, \emph{Banach lattices and positive operators},
  Springer-Verlag, New York, 1974, Die Grundlehren der mathematischen
  Wissenschaften, Band 215.

\end{thebibliography}

\end{document}